\documentclass[11pt]{article}
\usepackage[english]{babel}
\usepackage[a4paper,top=2cm,bottom=2cm,left=3cm,right=3cm,marginparwidth=1.75cm]{geometry}
\usepackage{amsmath,amsthm}
\usepackage{amssymb}
\usepackage{tikz}
\usepackage{siunitx}
\PassOptionsToPackage{hyphens}{url}
\usepackage{hyperref}
\usepackage{cleveref}
\usepackage[utf8]{inputenc}
\usepackage{amscd}
\usepackage{mathrsfs}
\usepackage{enumerate}
\usepackage{enumitem}
\usepackage{marginnote}
\usepackage{color}
\usepackage{xspace}
\usepackage{pgfplots}
\newcommand{\R}{\mathbb{R}}

\pgfplotsset{compat=1.18}
\numberwithin{equation}{section}
\newtheorem{theo}{Theorem}[section]
\newtheorem{prop}[theo]{Proposition}
\newtheorem{lemm}[theo]{Lemma}

\newtheorem{defi}[theo]{Definition}
\newtheorem{rema}[theo]{Remark}

\title{Convergence rates for a rotating MHD system with evanescent viscosity}
\author{Fr\'ed\'eric Charve\footnote{Univ Paris Est Creteil, Univ Gustave Eiffel, CNRS, LAMA UMR8050, F-94010 Creteil, France. E-mail: frederic.charve@u-pec.fr}, Zhuoya YAO\footnote{Institut Camille Jordan,\, Université Claude-Bernard-Lyon-I;\ LAMA,\,Univ Paris Est Creteil. E-mail: yao@math.univ-lyon1.fr}}
\date{}
\begin{document}
\maketitle
\begin{abstract}
In this article, we focus on a MHD model that takes into account the Coriolis force. In this case, the strength of the rotation is measured by the Rossby number $\varepsilon>0$. Under the assumption that the rotation is strong (that is, when $\varepsilon$ goes to zero) and the kinematic viscosity goes to zero as $\varepsilon^\alpha$ $(\alpha>0)$, we first establish the global convergence of weak solutions, and then prove the global existence and convergence of strong solutions. In particular, the Strichartz estimates used to obtain explicit convergence rates depend on the above parameter $\alpha$ and lead us to identify an admissible interval for $\alpha$.
\end{abstract}


\section{Introduction} 
\subsection{Introducing the MHD equations}

The magnetohydrodynamics (MHD) equations describe the coupling between electromagnetic fields and the dynamics of a conductive fluid, such as plasmas, liquid metals, or electrolytes, providing a framework for studying the behavior of conductive fluids under the influence of electric and magnetic fields. 
The structure of the magnetohydrodynamic (MHD) system can be viewed as the Navier–Stokes equations for fluid motion combined with the magnetic induction equation for the evolution of the magnetic field, with the two equations nonlinearly coupled through the Lorentz force in the momentum equation and the induction term in the magnetic equation, so that the velocity field and magnetic field continuously influence each other’s dynamics.

There is a vast literature devoted to the study of various models taking into account the MHD coupling, especially to the study of the classical MHD system.
\begin{equation}
	\label{eq:MHD}
	\left\{
	\begin{aligned}
		&\partial_t u - \nu \Delta u + u\cdot\nabla u - B \cdot \nabla B = - \nabla P,
		\\
		&\partial_t B - \nu' \Delta B + u\cdot \nabla B - B\cdot\nabla u = 0,
		\\
		&\operatorname{div}  u = 0,\quad \operatorname{div} B = 0.
	\end{aligned}
	\right.
\end{equation}
The classical MHD system is parabolic and due to a good symmetry of the nonlinear terms, one can easily prove the same type of energy inequality which enables the classical construction of Leray weak solutions,
\begin{equation}
	\frac12 \left(\|{u(t)}\|^2_{L^2}+\|{B(t)}\|^2_{L^2} \right)+\int_0^t(\nu\|\nabla u(s)\|^2_{L^2}+\nu^{'} \|\nabla B(s)\|^2_{L^2})ds \leq \frac12 \left(\|u(0)\|^2_{L^2} +\|B(0)\|^2_{L^2}\right).
\end{equation}

In \cite{DL1972}, Duvaut and Lions proved the global existence of weak solutions to \eqref{eq:MHD} if the initial data belongs to  $L^2(\R^3)$ and the local existence of a unique strong solution to \eqref{eq:MHD} if the initial data belongs to $H^m(\R^3)$ with $m\geq 3$. In \cite{ST83}, weak and strong solutions of \eqref{eq:MHD} are studied for initial data in $H^1(\R^3)$. For more results about the solutions of \eqref{eq:MHD}, we refer to \cite{AP08, LZ14}.

Next we introduce MHD flow in the earth's core, which is believed to support a self-excited dynamo process generating the earth's magnetic field
\begin{equation}\label{original}
\begin{cases}
\partial_t u + u \cdot \nabla u 
+ \dfrac{\nabla p}{\varepsilon}
+ \dfrac{e \times u}{\varepsilon}
- \dfrac{E}{\varepsilon} \Delta u
= \dfrac{\Lambda}{\varepsilon \theta} \, (\nabla \times B) \times B, \\[8pt]
\partial_t B
= \nabla \times (u \times B)
+ \dfrac{1}{\theta} \Delta B, \\[6pt]
\operatorname{div} B = 0, 
\qquad
\operatorname{div} u = 0,
\end{cases}
\end{equation}
and in $\Omega^c$, we have
\begin{equation}
\nabla \times B = 0,\,\, 
\nabla \times E = -\partial_t B,\,\,
\operatorname{div} E = 0, \,\,
\operatorname{div} B = 0,
\end{equation}
where $E,\varepsilon,\Lambda,\theta$ respectively denote the Ekman, Rossby, Elsasser, and the magnetic Reynolds numbers. As explained in \cite{DDG99}, it is also possible to split the magnetic field $B$ into two parts: first, a large-scale, time-independent field $B_{0}=e^{'}$ and second, a scaled perturbation $b$ such that:
\begin{equation}
    B=e^{'}+\theta b.
\end{equation}
So that \eqref{original} becomes in $\Omega$
\begin{equation}\label{modiMHD}
\begin{cases}
\partial_tu - u\cdot \nabla u +\frac{\nabla P}{\varepsilon}-\frac{E}{\varepsilon}\Delta u+\frac{e\times u}{\varepsilon} =\frac{\Lambda}{\varepsilon} curl(b)\times e^{'}+\frac{\Lambda \theta}{\varepsilon}\operatorname{curl}(b)\times b, \\
        \partial_t b-+ u\cdot \nabla b= b \cdot \nabla u+\frac{\operatorname{curl}(u\times e^{'})}{\theta}+\frac{\Delta b}{\theta}, \\
        \operatorname{div}(u_\varepsilon) = \operatorname{div}(b_\varepsilon) = 0,
    \end{cases}
\end{equation}
and in $\Omega^{c}$
\begin{equation*}
\operatorname{curl}b=0,\,\operatorname{curl}E=-\theta\partial_t b,\ \operatorname{div}(E) =\operatorname{div}(b)= 0.
\end{equation*}

Such models are interesting to describe the evolution of the plasmas in the core of the Earth or stars (see for instance \cite{DDG99} and the references therein for physical motivations). In \cite{DDG99},  Desjardins, Dormy and Grenier proved rigorously that if the Reynolds number defined on boundary-layer characteristics is smaller than a critical value, the boundary layer is nonlinearly stable. They also proved that the normal component of the magnetic field increases the critical Reynolds number for instability.

In \cite{VSN17}, Ngo considered \eqref{modiMHD} with no vertical diffusion and only a small horizontal diffusion, in the case where $b$ is a small perturbation of $e_3$ and proved that for $\varepsilon$ close to $0$, \eqref{modiMHD} is globally well-posed for initial data independent of $\varepsilon$, using Strichartz estimates on the associated wave system. Also using Strichartz estimates in \cite{AKL21}, Ahn, Kim and Lee proved similar results as in \cite{VSN17} for system \eqref{original} (see also \cite{KJ22, TY2424}). Finally, let us mention \cite{FCVSN25} where the authors consider the 3D rotating MHD system when the initial velocity and magnetic field both feature 2D-parts (i.-e. are sum of a function depending on $(x_1,x_2,x_3)$ and a function only depending on $(x_1,x_2)$, both with three components) and prove for weak and strong solutions that the limit system is a 2D-MHD system with six components.

Moreover, $e$ and $e^{'}$ are fixed vectors. As in \cite{DDG99}, choosing $(E,\Lambda,\theta)=(\varepsilon^2, \varepsilon,1)$, and in the canonical basis $(e_1,e_2,e_3)$ of $\R^3$, $e=e^{'}=-e_3$ leads to the following:
\begin{equation}\label{MHD1}
    \begin{cases}
       \partial_t u_\varepsilon - \varepsilon \Delta u_\varepsilon 
   + u_\varepsilon \cdot \nabla u_\varepsilon
   - \operatorname{curl} b_\varepsilon \wedge b_\varepsilon
   + \operatorname{curl} b_\varepsilon \wedge e
   + \frac{1}{\varepsilon}u_\varepsilon \wedge e_3= 0, \\
\partial_t b_\varepsilon - \Delta b_\varepsilon 
   + u_\varepsilon \cdot \nabla b_\varepsilon 
   - b_\varepsilon \cdot \nabla u_\varepsilon
   + \operatorname{curl}(u_\varepsilon \wedge e_3) = 0, \\
\operatorname{div}(u_\varepsilon) = \operatorname{div}(b_\varepsilon) = 0.
    \end{cases}
\end{equation}
In this case,  \cite{BIM05} obtained both in the case of the torus $\mathbb{T}^3$ or the whole space $\R^3$, global existence of weak solutions, local existence of strong solutions (global if the data have a norm smaller than $c\varepsilon$, where $c > 0$ is some constant), and
asymptotics when $\varepsilon$ goes to zero.

Let us end this section with a very recent result in \cite{HO26}, where the author proves long-time existence in the case $\alpha=1$ for constant initial data (for any $T$, if epsilon is small enough, depending on T, then the solution exists on [0,T]) and local-in-time explicit convergence rates in terms of epsilon.

In our article we prove global-existence (stronger proposition) for very large initial data and give explicit global-in-time convergence rates in terms of epsilon.
\subsection{What we focus on and our results}
Inspired by the work \cite{DDG99}, several works devoted to Systems \eqref{original} and \eqref{modiMHD} (we refer to \cite{AKL21,B05, BIM05, BGM05,FCVSN25}) investigated the case where $b$ is a small perturbation of $e_3$, and studied for an evanescent kinematic viscosity the convergence of strong solutions of \eqref{MHD1} when $\varepsilon \to 0$, under different choices of the parameter triplet $(E,\theta,\Lambda)$:
\begin{equation}
(E,\theta,\Lambda)\in\Big\{(\varepsilon^2,1,\varepsilon),\;(\varepsilon^2,\varepsilon^{-\frac12},\varepsilon^{\frac32}),\;
(\varepsilon,1,\varepsilon),\;\left(\nu\varepsilon,\frac{1}{\nu'},\frac{\varepsilon}{\nu'}\right)\Big\}.
\end{equation}

In this paper, we consider the case of MHD flows in fast rotation. Recall that the evolution of a vector field in an absolute reference frame and a rotating frame attached to the Earth (of angular velocity $\varepsilon^{-1}$) is given by some well-known formula.
Our aim is to study System \eqref{original} in the case $E=\varepsilon^{\alpha+1}, \Lambda=\varepsilon, \theta=1, e=-e_3$ and provide more precise results with an evanescent viscosity. This case corresponds to consider the following evanescent viscosity: $\nu=\varepsilon^\alpha$ (where $\alpha \in [0,1)$),
\begin{equation}\label{MHDalpha}
    (MHD^\alpha) \, 
    \begin{cases}
        \partial_tu_\varepsilon - \varepsilon^{\alpha} \Delta u_\varepsilon + u_\varepsilon \cdot \nabla u_\varepsilon - b_\varepsilon \cdot \nabla b_\varepsilon -\frac{1}{\varepsilon}e_3\land u_\varepsilon= -\nabla p_\varepsilon, \\
        \partial_t b_\varepsilon- \Delta b_\varepsilon + u_\varepsilon \cdot \nabla b_\varepsilon - b_\varepsilon \cdot \nabla u_\varepsilon= 0, \\
        \operatorname{div}(u_\varepsilon) =\operatorname{div}(b_\varepsilon) = 0,\\
        (u_\varepsilon,b_\varepsilon)|_{t=0}=(u_{0,\varepsilon},b_{0,\varepsilon}).
    \end{cases}
\end{equation}

One of the difficulties of \eqref{MHDalpha} is that the magnetic field $b_\varepsilon$ does not a priori possess any dispersion. Moreover there are far fewer results concerning the asymptotics of \eqref{MHDalpha} when $\varepsilon \to 0$.

Let us now state the main results of this article, which deal with both weak and strong solutions for System \eqref{MHDalpha}. First, for the weak solutions, in order to simplify, we consider initial data independent of the small parameter: $(u_0,b_0)$, which corresponds to the following energy inequality
\begin{equation}\label{en1}
\|u_\varepsilon\|_{L^2}^2 +\|b_\varepsilon\|_{L^2}^2+2\varepsilon^\alpha \int_0^t \|\nabla u_\varepsilon\|_{L^2}^2 + 2 \int_0^t \|\nabla b_\varepsilon\|_{L^2}^2 \leq\|u_0\|_{L^2}^2+\|b_0\|_{L^2}^2.
\end{equation}

\begin{theo}[Weak solutions]\label{mainthe1}
For any $\alpha\in[0, \frac{1}{14})$, and $r\in(2,r_1(\alpha))\cup(r_2(\alpha),3)$, and any $u_0,b_0\in L^2(\R^3)$, there exist some constants $C, A_r, B_r, M_{p,r}, N_r$, such that any weak solution $(u_{\varepsilon}, b_{\varepsilon})$ satisfies for any $t\geq0$:
\begin{equation}
    \|u_\varepsilon\|_{L^2_tL^{r}} \leq C \max\{1,t\}^{A_{r}} \,\varepsilon^{{B_{r}}},
\end{equation}
and for any $p\in(1,p(\alpha)),$ there exists an interval 
 $\mathcal{I}_{p} \subset(2,r_1(\alpha))$, such that for any $r\in\mathcal{I}_{p},$ and any $t\geq0$:
\begin{equation}
    \|b_\varepsilon-e^{t\Delta}b_0\|_{L^p}\leq C\max\{1,t\}^{M_{p,r}}\varepsilon^{N_{r}},
\end{equation}
where $$
  r_1(\alpha)=\frac{3\alpha-2}{4\alpha-1},\  r_2(\alpha)=\frac{3\alpha+3}{4\alpha+1},\ p(\alpha)=\frac{6(2-3\alpha)}{5(2-\alpha)},$$
 $$A_{r}=\frac{5}{2r}-\frac{3}{4},\ B_{r}=\min\{\frac{r-2}{2r},\,\frac{3-r}{2r}\}-\alpha(2-\frac{3}{2r}),$$
and $$\,M_{p,r}=\frac{5}{2r}-\frac32+\frac{3}{2p},\,\,N_r=\frac{r-2}{2r}-\alpha(2-\frac{3}{2r})  
.$$
\end{theo}
\begin{rema}
    Compared to \cite{FCVSN25}, we are able to also provide a strong convergence rate for $b_\varepsilon-e^{t\Delta}b_{0,\varepsilon}$.
\end{rema}
Secondly, let us state a simplified version of our result for strong solution (we refer to Theorem \ref{mainthe22} for a precise statement):
\begin{theo}[Strong solutions]\label{mainthe2}
For any $\delta\in(0,\frac13)$(extra regularity), $\gamma\in[0,\frac\delta2), \alpha\in[0,\frac{1}{8})$ such that $$\alpha(\frac32-\frac\delta2)+\gamma<\frac\delta2,$$for any $ C_0>0$, and any $k\in(0,1),$ there exists $\varepsilon_0, m_0$ (both of them depending on $\delta,\gamma,\alpha,C_0,k$) such that for any $\varepsilon\leq \varepsilon_0,$ any initial data such that
\begin{equation} 
\|u_{0,\varepsilon}\|_{\dot H^{\tfrac12-\delta}\cap \dot H^{\tfrac12+\delta}}
\leq C_0\varepsilon^{-\gamma},\ \ 
\|b_{0,\varepsilon}\|_{H^{\tfrac12+\delta}}
\leq m_0 |\ln \varepsilon|^{\tfrac12},
\end{equation}
then the unique strong solution $(u_\varepsilon, b_\varepsilon)$ of System \eqref{MHDalpha} is global and we have the following estimate,
\begin{equation}
\|(\varepsilon^{\frac\alpha2}u_\varepsilon,b_\varepsilon-e^{t\Delta}b_{0,\varepsilon})\|_{L^2L^{\infty}}\leq C_0\varepsilon^{k(\frac\delta2-\gamma)-\alpha(\frac12-\frac\delta2)}.
\end{equation} 
\end{theo}
\begin{rema}
 \begin{enumerate}
  \item In the previous works \cite{B05, BIM05,BGM05}, the authors assumed the initial data to be very small (of size $\varepsilon$) for global existence ; we are able to obtain global existence for large data of size $\varepsilon^{-\gamma}$ for the initial velocity and $|\ln \varepsilon|^{\frac12}$ for the initial magnetic field in Theorem \ref{mainthe2}.
  \item We really need extra regularity (if $\delta=0,$ we are able to prove convergence when $\nu$ is constant and data independent of $\varepsilon$ but cannot provide any convergence rates. [see \cite{FC04, HO26}].)
  \item In the present article, we also improve the size of initial data $b_{0,\varepsilon}$ compared to \cite{FCVSN25}, namely $\|b_{0,\varepsilon}\|_{H^{\frac12+\delta}}$ is now of size $|\ln \varepsilon|^{\frac12}$.
  \item Our $k$ aims to be close to $1$ in order to reach the optimal speed of convergence (the closer it is to $1$, the smaller becomes $\varepsilon_0$ ).
 \end{enumerate}
\end{rema}

\subsection{Outline of the article}

In section 2, we introduce the notation used in this paper and the specific tools that we employ in Sections 3 and 4.

In section 3, we first derive the global-in-time estimate for velocity of weak solutions using Strichartz estimates; followed by a simpler local-in-time version, in which we give an explicit convergence rate depending on $\alpha.$  Next we show the local convergence of the magnetic part, and also establish the relationship between the parameters in the velocity and in the magnetic-field cases.

In section 4, we present the global existence and convergence theorem for very large ill-prepared initial data by isolating the linear part in \eqref{MHDalpha} that provides dispersion. In this proof, compared to \cite{FCVSN25} we use non-local operators (as in \cite{FC20, FCsharper}) in order to deal with several difficult terms. Indeed, several harmless terms in \cite{FCVSN25} become difficult in the case of an evanescent viscosity.
\section{Preliminaries}
\subsection{Basic knowledge for general Sobolev and Besov spaces}
\begin{defi}[Inhomogeneous Sobolev Space]
Let $s \in \R$. The Sobolev space $H^s(\R^d)$ consists of tempered distributions $u$ whose Fourier transform belongs to $L^2_{\mathrm{loc}}(\R^d)$ and satisfies
\begin{equation*}
 \|u\|_{H^s}^2 
\stackrel{\mathrm{def}}{=} 
\int_{\R^d} (1 + |\xi|^2)^s |\hat{u}(\xi)|^2 \, d\xi < \infty.  \end{equation*}
\end{defi}

\begin{defi}[Homogeneous Sobolev Space]
Let $s \in \R$. The homogeneous Sobolev space $\dot{H}^s(\R^d)$ is the space of tempered distributions $u$ on $\R^d$ whose Fourier transform belongs to $L^2_{\mathrm{loc}}(\R^d)$ and satisfies
\begin{equation*}
\|u\|_{\dot{H}^s}^2 
\stackrel{\mathrm{def}}{=} 
\int_{\R^d} |\xi|^{2s} |\hat{u}(\xi)|^2 \, d\xi < \infty. \end{equation*}
\end{defi}
Then we follow basic ideas from the Littlewood-Paley theory. We denote by
$ \varphi\in\mathcal{S}(\R^{3}) $ a radially symmetric function supported in
$ \{\xi\in\R^{3}:\frac{3}{4}\leq|\xi|\leq\frac{8}{3}\} $ such that
\begin{equation*}
\sum\limits_{j\in\mathbb{Z}}\varphi(2^{-j}\xi)=1 \quad \text{for all} \quad \xi\neq 0.
\end{equation*}
We also introduce the following functions
\begin{equation*}
\varphi_{j}(\xi)=\varphi(2^{-j}\xi) \quad  \text{and}\quad \psi_{j}(\xi)=\sum\limits_{k\leq j-1}\varphi_{k}(\xi).
\end{equation*}
Now we define the standard localization operators for every $j\in \mathbb{Z}$:
\begin{equation*}
\dot{\Delta}_{j}f=\varphi_{j}(D)f, \quad S_{j}f=\sum\limits_{k\leq j-1}\Delta_{k}(D)=\psi_{j}(D)f,\quad \dot{S}_{j}f=\sum\limits_{k\leq j-1}\dot{\Delta}_{k}(D).
\end{equation*}
It is then easy to verify the following identities:
\begin{align*}
\dot{\Delta}_{j}\dot{\Delta}_{k}f=0\quad &\text{if} \quad |j-k|\geq 2,\\
\dot{\Delta}_{j}(S_{k-1}f\dot{\Delta}_{k}f)=0 \quad&\text{if} \quad |j-k|\geq 5.
\end{align*}
We now define the homogeneous Besov spaces based on the above dyadic decomposition.

\begin{defi}
Let $s \in \R$ and $(p,q) \in [1,\infty]^2$.
The homogeneous Besov space $\dot{B}^s_{p,q}(\R^3)$ is defined as the set of tempered distributions $u$ such that $\|S_j u\|_{L^\infty} \underset{j\rightarrow -\infty}{\longrightarrow}0$ and
\[
\|u\|_{\dot{B}^s_{p,q}} 
\stackrel{\mathrm{def}}{=} 
\left( \sum_{j \in \mathbb{Z}} 
2^{jsq} \|\dot{\Delta}_j u\|_{L^p}^q
\right)^{1/q} < \infty,
\]
with the usual modification when $r=\infty$:
\[
\|u\|_{\dot{B}^s_{p,\infty}} 
= \sup_{j \in \mathbb{Z}} 
2^{js} \|\dot{\Delta}_j u\|_{L^p}.
\]
\end{defi}
In the function space defined above, as we excluded polynomials, $S_j u=\dot{S}_j u$. We will also use spaces that are slight modification of $L^{r}  \dot{B}^s_{p,q}$: namely the Chemin-Lerner time-space Besov spaces for which the integration in time is performed before the summation with respect to the frequency decomposition index:
\begin{defi}
    For $s,t\in \R$ and $a,b,c\in[1,\infty]$, we define the following norm
    \begin{equation*}
      \|u\|_{\widetilde{L}^{r}  \dot{B}^s_{p,q}}=\|(2^{js}\|\dot{\Delta}_j u\|_{L^r_tL^p})\|_{l^q(\mathbb{Z})}. 
    \end{equation*}
\end{defi}
The space $\widetilde{L}^{r}  \dot{B}^s_{p,q}$ is defined as the set of tempered distribution $u$ such that $\lim\limits_{j\longrightarrow-\infty}S_j u=0$ in $L_t^{r}L^{\infty}(\R^d)$ and $\|u\|_{\widetilde{L}^{r}  \dot{B}^s_{p,q}}<\infty.$
\begin{rema}
    $L_t^{r}L^{\infty}(\R^d)$ and $\widetilde{L}_t^{r}L^{\infty}(\R^d)$ refer to a time integration on $[0,t],$ and if we integrate on $\mathbb{R}_{+}$ in time, the spaces are denoted as $L^{r}L^{\infty}(\R^d)$ and  $\widetilde{L}^{r}L^{\infty}(\R^d).$ 
\end{rema}
Let us also recall the following propositions and lemmas:
\begin{prop}\label{defbesov}
    For all $r,p,q\in [1,\infty]$ and $s\in \R$:
\begin{align*}
    \text{if}\,\, \ r\leq q,\ \text{for any}\,\,u \in L^{r}  \dot{B}^s_{p,q},\ \|u\|_{\widetilde{L}^{r}  \dot{B}^s_{p,q}}\leq \|u\|_{L^{r}  \dot{B}^s_{p,q}};\\
    \text{if}\,\,\ r\ge q,\ \text{for any}\,\,u \in \widetilde{L}^{r}  \dot{B}^s_{p,q},\ \|u\|_{\widetilde{L}^{r}  \dot{B}^s_{p,q}}\ge\|u\|_{L^{r}  \dot{B}^s_{p,q}}.
\end{align*}
\end{prop}
\begin{prop}\label{Classicalinjections}{(Classical injections): We have:
$$
 \begin{cases}
\dot{B}_{p,1}^0 \hookrightarrow L^p, &\mbox{for any } p\geq 1 ,\\
\dot{B}_{p,2}^0 \hookrightarrow L^p, & \mbox{for any } p\in[2,\infty[,\\
\dot{B}_{p,p}^0 \hookrightarrow L^p,& \mbox{for any } p\in[1,2].
\end{cases}
$$
}
\end{prop}
\begin{prop}\label{interpolation} (\cite{BCD11})
 \sl{For any $\alpha, \beta>0$ there exists a constant $C_{\alpha, \beta}>0$ such that for any $u\in \dot{H}^{s-\alpha}(\R^d) \cap \dot{H}^{s+\beta}(\R^d)$, then $u\in\dot{B}_{2,1}^s(\R^d)$ and:
\begin{equation*}
 \|u\|_{\dot{B}_{2,1}^s} \leq C \|u\|_{\dot{H}^{s-\alpha}}^{\frac{\beta}{\alpha + \beta}} \|u\|_{\dot{H}^{s+\beta}}^{\frac{\alpha}{\alpha+ \beta}}.
\end{equation*}
 }
 \end{prop}
\begin{prop}\label{Generalproductlaws}
 \sl{There exists a constant $C>0$ such that for any $s,t<\frac32$ with $s+t>0$ and any $u\in \dot{H}^s(\R^3)$, $v\in \dot{H}^t(\R^3)$, then $uv\in \dot{H}^{s+t-\frac32}(\R^3)$ and we have:
 $$
 \|uv\|_{\dot{H}^{s+t-\frac32}(\R^3)} \leq C \|u\|_{\dot{H}^s(\R^3)} \|v\|_{\dot{H}^t(\R^3)}.
 $$
 }
\end{prop}
\begin{lemm}\label{twospace}
For two Banach spaces \(X\) and \(Y\), the norm in the sum space \(X+Y\) is defined by
\[
\|f\|_{X+Y}=\inf_{f=f_1+f_2}\big(\|f_1\|_X+\|f_2\|_Y\big).
\]
\end{lemm}

\begin{lemm}\label{Sobolevproductlaw}
Let $s>0$. Then for any smooth functions $f,g$ on $\R^3$,
\begin{equation*}
\|fg\|_{H^s}
\le C\big(
\|f\|_{L^\infty}\|g\|_{H^s}
+
\|g\|_{L^\infty}\|f\|_{H^s}
\big).
\end{equation*}
\end{lemm}

\subsection{Strichartz estimates for the rotating fluids system}
Let us explain how we can study this rotating magnetohydrodynamic equation \eqref{MHDalpha}.
First we recall previous results focusing on the asymptotics when the Rossby number goes to zero. In the case where there is no magnetic field (i.-e. $b_\varepsilon \equiv 0$), \eqref{MHDalpha} reduces to the well-known rotating fluids system
\begin{equation*}\label{RF}
(RF^\varepsilon)\quad
\left\{
\begin{aligned}
\partial_t u^\varepsilon
+ u^\varepsilon \cdot \nabla u^\varepsilon
- \nu \Delta u^\varepsilon
- \frac{e_3\wedge u^\varepsilon}{\varepsilon}
&= - \nabla p^\varepsilon
&& \text{in } \R^3, \\
\operatorname{div} u^\varepsilon
&= 0
&& \text{in } \R^3, \\
u^\varepsilon|_{t=0}
&= u_0.
\end{aligned}
\right.
\end{equation*}

Given the vast literature developed in the past, we briefly recall some of the most important results that have been established. First we remark that the only way to deal with the singular perturbation $\frac{1}{\varepsilon} e_3\wedge u^\varepsilon$ is to compensate it with the pressure term. So when $\varepsilon \to 0$, the pressure no longer depends on the vertical variable $x_3$. In order to better understand the effect of the fast rotation, it is useful to  firstly consider the linear system
\begin{equation*}
	\partial_t u^{\varepsilon} - \frac{1}{\varepsilon} \mathbb{P}(e_3\wedge u^\varepsilon) = 0,
\end{equation*}
where $\mathbb{P}$ denotes the Leray projection from $L^2$ onto the divergence-free vector fields. This equation highlights the dispersion induced by the Coriolis force: it generates rapidly oscillating modes with characteristic frequency of order $\varepsilon^{-1}$. More precisely, these oscillations correspond to wave packets propagating at high speed. In the framework of geophysical fluid dynamics, they are known as Rossby waves. A key feature of these waves is their dispersive behavior: they propagate away from the region of interest and thus effectively transport energy to infinity.

And then it will be useful in the next section to define the solutions of the linear equations associated with $(RF^\varepsilon)$, where both $F_{ext}$ and $v_0$ are divergence-free:
\begin{equation}\label{FRF}
(FRF^\varepsilon)\quad
\left\{
\begin{aligned}
\partial_t v_F^\varepsilon
- \nu \Delta v_F^\varepsilon
+\mathbb{P}(\frac{1}{\varepsilon}e_3\wedge v_F^\varepsilon)
&= F_{ext}
&& \text{in } \R^3, \\
v_F^\varepsilon|_{t=0}
&= v_0.
\end{aligned}
\right.
\end{equation}
We can use Strichartz estimates to obtain the following kind of estimates:
\begin{equation*}
\|v_F^\varepsilon\|_{\square}\leq \varepsilon^{\square}.
\end{equation*}

Among the various forms of the Strichartz estimates, we will use the following one, which is adapted from \cite{FC23} to our case where $\nu=\varepsilon^\alpha:$
\begin{lemm}[\cite{FC23}]\label{eststri1} For any $d \in \R$, $r \geq 2$, $q \geq 1$, $\theta \in [0,1]$ and $p \in \left[1,\frac{2}{\theta(1- \frac{2}{r})}\right]$, there exists a constant $C $ such that for any $v_F^\varepsilon$ solving System \eqref{FRF} with initial data $v_0$ and external force $F_{\text{ext}}$ both having zero divergence, we have
\begin{equation*}
\left\| |D|^d v_F^\varepsilon \right\|_{\widetilde{L}^p_t \dot{B}^0_{r,q}} 
\leq C\varepsilon^{\frac{\theta}{2}\left(1-\frac{2}{r} \right)-\alpha({\frac{1}{p} - \frac{\theta}{2}\left(1 - \frac{2}{r} \right)})}
\left( 
\left\| v_0 \right\|_{\dot{B}^{\sigma_1}_{2,q}} 
+ 
\left\| F_{\text{ext}} \right\|_{\widetilde{L}^1_t \dot{B}^{\sigma_1}_{2,q}} 
\right),    
\end{equation*}
where 
$\sigma_1 = d + \frac{3}{2} - \frac{3}{r} - \frac{2}{p} + \theta(1 - \frac{2}{r}).$\end{lemm}

\section{Weak solutions}

In this part we prove Theorem \ref{mainthe1}. Easily adaptating what we know from \cite{BIM05}: for any $\varepsilon$ fixed, for initial data $(u_0,b_0)\in (L^2(\R^3))^2$, there exists at least one global weak solution (in the sense of Leray), $(u_\varepsilon,b_\varepsilon)$ which satisfies the energy estimates \eqref{en1}.

\subsection{For velocity part}

First, applying the Leray projector to the first equation of \eqref{MHDalpha} gives:
\begin{center}$  \partial_tu_\varepsilon - \varepsilon^{\alpha} \Delta u_\varepsilon-\mathbb{P}(\frac{1}{\varepsilon}e_3\land u_\varepsilon) = F,
$\end{center}
where
\begin{center}
$F=-\mathbb{P}( u_\varepsilon \cdot \nabla u_\varepsilon) + \mathbb{P}(b_\varepsilon \cdot \nabla b_\varepsilon)$.
\end{center}
Subsequently, using a superproposition principle as in \cite{FC23, FCVSN25}, we split $u_\varepsilon=u_{\varepsilon,1}+u_{\varepsilon,2}$ and use Strichartz estimates from Lemma \ref{eststri1} to establish global estimates for  $u_{\varepsilon,1},u_{\varepsilon,2}$ defined as follows
\begin{equation*}\label{S1}
\begin{aligned}
\begin{cases}
\partial_tu_{\varepsilon,1} - (\varepsilon^{\alpha}\Delta+ \mathbb{P}(\frac{1}{\varepsilon}e_3\land))u_{\varepsilon,1} = 0, \\
u_{\varepsilon,1}|_{t=0} = u_{0},
\end{cases}\end{aligned}  
\begin{aligned}
        \begin{cases}
        \partial_tu_{\varepsilon,2} - (\varepsilon^{\alpha}\Delta+ \mathbb{P}(\frac{1}{\varepsilon}e_3\land)) u_{\varepsilon,2} = F, \\
        u_{\varepsilon,2}|_{t=0} = 0.
    \end{cases}
\end{aligned}
\end{equation*}
\begin{prop}\label{ranger1}
If $u_{\varepsilon,1}$ is as defined, then for any 
$r_1>2$, $\theta_{r_1}=\min\{1,\frac{r_1+6}{2(r_1-2)}\}$, and $p_{r_1}=\max\{1,\frac{4}{5(1-\frac{2}{r_1})}\},$ we have
\begin{equation*}
\|u_{\varepsilon,1}\|_{\widetilde{L}^{p_{r_1}} \dot{B}^{0}_{r_1,2}} \leq C \varepsilon^{(\frac{\theta_{r_1}}{2}-\frac{3\alpha}{4})(1 - \frac{2}{r_1})}\|u_{0}\|_{L^2}.
\end{equation*}
\end{prop}
\begin{proof}
Applying Lemma \ref{eststri1}, we know that 
\begin{equation*}
\begin{aligned}
\|u_{\varepsilon,1}\|_{\widetilde{L}^{p_{r_1}} \dot{B}^{0}_{r_1,2}}\leq C \varepsilon^{(\frac{\theta_{r_1}}{2}-\frac{3\alpha}{4})(1 - \frac{2}{r_1})}\|u_{0}\|_{\dot{H}^\sigma},   
\end{aligned}
\end{equation*}
where $\sigma=0+\frac{3}{2}-\frac{3}{r_1}-\frac{2}{p_{r_1}}+\theta_{r_1}(1-\frac{2}{r_1}),$ and we choose $q=2.$ Then
\begin{equation*}
\sigma=0\Longleftrightarrow0=\frac{3}{2}-\frac{3}{r_1}-\frac{2}{p_{r_1}}+\theta_{r_1}(1-\frac{2}{r_1}),
\end{equation*}
so we can get
\begin{equation}\label{forp}
    \frac{1}{p_{r_1}}=\frac{3}{4}-\frac{3}{2r_1}+\frac{\theta_{r_1}}{2}(1-\frac{2}{r_1}).
\end{equation}
Because parameter $p_{r_1}$ has to satisfy the condition in Lemma \ref{eststri1}
\begin{equation}\label{pequi}
    \frac{\theta_{r_1}}{2}(1-\frac{2}{r_1})\leq\frac{1}{p_{r_1}}\leq 1\\\Longleftrightarrow\begin{cases}
        0\leq\frac{3}{4}-\frac{3}{2r_1},\\
        (\frac{\theta_{r_1}}{2}+\frac{3}{4})(1-\frac{2}{r_1})\leq 1.
    \end{cases}
\end{equation}
\begin{itemize}
    \item If $\theta_{r_1}\in[0,\frac12]$, \eqref{pequi} is true for $r_1\in(2,\infty)$; 
    \item if $\theta_{r_1}\in(\frac12,1]$, $r_1\leq \frac{2(3+2\theta_{r_1})}{2\theta_{r_1}-1}=2+\frac{8}{2\theta_{r_1}-1}$ \eqref{pequi} is also true for  $r_1\in[2,\frac{2(3+2\theta_{r_1})}{2\theta_{r_1}-1}].$
\end{itemize}
So for any $\theta_{r_1}\in[0,1],r_1\in(2,\infty),$ we have 
\begin{equation*}
    r_1-2\leq\frac{8}{2\theta_{r_1}-1}\Longleftrightarrow
2\theta_{r_1}-1\leq \frac{8}{r_1-2}\Longleftrightarrow
\theta_{r_1}\leq \frac{r_1+6}{2(r_1-2)}.
\end{equation*}
Because of this restriction on $\theta_{r_1}$ in Lemma \ref{eststri1}, we choose \begin{equation*}
    \theta_{r_1}=\min\{1,\frac{r_1+6}{2(r_1-2)}\}.
\end{equation*}
So thanks to \eqref{forp},
\begin{equation*}
    p_{r_1}=\max\{1,\frac{4}{5(1-\frac{2}{r_1})}\}.
\end{equation*}
Finally, we obtain 
\begin{equation*}
\|u_{\varepsilon,1}\|_{\widetilde{L}^{p_{r_1}}  \dot{B}^0_{r_1,2}}\leq C \varepsilon^{(\frac{\theta_{r_1}}{2}-\frac{3\alpha}{4})(1-\frac{2}{r_1})}\|u_0\|_{L^2}.
\end{equation*}
\end{proof}

\begin{prop}\label{rangr2}
If $u_{\varepsilon,2}$ is as defined, then for any $r_2>2,\theta_{r_2}=\min\{1,\frac{3}{r_2-2}\}$ and $p_{r_2}=\max\{1,\frac{2}{3-\frac{5}{r_2}}\}$,
\begin{equation*}
\begin{aligned}
\|u_{\varepsilon,2}\|_{\widetilde{L}^{p_{r_2}} \dot{B}^{0}_{r_2,2}}\leq C \varepsilon^{\frac{\theta_{r_2}}{2}(1 - \frac{2}{r_2})-\alpha(2-\frac{3}{2r_2})}(\|u_0\|^2_{{L}^2}+\|b_0\|^2_{{L}^2}).   
\end{aligned}
\end{equation*}
\end{prop}
\begin{proof}
From \eqref{en1}, we know that
\begin{equation*}
\|F\|_{\widetilde{L}^1\dot{H}^{-\frac{1}{2}}}\leq C (1+\frac{1}{\varepsilon^\alpha})(\|u_0\|^2_{{L}^2}+\|b_0\|^2_{{L}^2}).\\
\end{equation*}
So we obtain
\begin{equation*}
\|u_{\varepsilon,2}\|_{\widetilde{L}^{p_{r_2}} \dot{B}^{0}_{r_2,2}}
\leq C \varepsilon^{\frac{\theta_{r_2}}{2}(1 - \frac{2}{r_2})-\alpha(2-\frac{3}{2r_2})}(\|u_0\|^2_{{L}^2}+\|b_0\|^2_{{L}^2}). \end{equation*}
The proof of how to choose $\theta_{r_2},\ p_{r_2}$ is similar to the previous one and is therefore omitted.
\end{proof}
In order to obtain estimates for the magnetic field $b_\varepsilon$, we need to handle with products. Such computations are easier if we manipulate Lebesgue norms in space instead of homogeneous Besov norms. This is why we need handier estimates for the velocity, and to get such norms, we have to consider local-in-time norms, and to apply Propositions \ref{defbesov} and \ref{Classicalinjections}, which require $p_{r_1}\geq2,\ p_{r_2}\geq2$. Let us now  state the following two propositions:
\begin{prop}\label{range26}For any 
$2<r_1<6$, if we define $\theta_{r_1}^{'}=\min\{1,\frac{6-r_1}{2(r_1-2)}\}$ and 
$p_{r_1}^{'}=\max\{2,\frac{4}{5(1-\frac{2}{r_1})}\}$, then
\begin{equation*}
\|u_{\varepsilon,1}\|_{L^{{p_{r_1}^{'}}} L^{r_1}}\leq 
C \varepsilon^{(\frac{\theta_{r_1}^{'}}{2}-\frac{3\alpha}{4})(1-\frac{2}{r_1})}
\|u_0\|_{L^2}.
\end{equation*}
\end{prop}
\begin{proof}
If we want to get the estimate in $L^pL^r$,
we have to use
\begin{equation*}
    \|u\|_{L^pL^r}\leq \|u\|_{L^p\dot{B}^0_{r,2}}\leq\|u\|_{\widetilde{L}^p\dot{B}^0_{r,2}},
\end{equation*}
which requires Propositions \ref{defbesov} and \ref{Classicalinjections}. Then we come back to Proposition \ref{ranger1}. And $p_{r_1}^{'}\ge 2$ means that \eqref{forp} should satisfy
\begin{equation}\label{thetar}
\frac{4}{3+2\theta_{r_1}^{'}-\frac{6+4\theta_{r_1}^{'}}{r_1}}\ge 2\Longleftrightarrow r_1 \leq 2+\frac{4}{2\theta_{r_1}^{'} +1}.
\end{equation} 
Because of $\theta_{r_1}^{'} \in[0,1]$, we know $2+\frac{4}{2\theta_{r_1}^{'}+1}\in [\frac{10}{3},6).$
So the range of $r_1$ is $(2,6).$
From \eqref{thetar} we know that $\theta_{r_1}^{'} \leq \frac{6-r_1}{2(r_1-2)}$, so we choose
$\theta_{r_1}^{'}=\min\{1,\frac{6-r_1}{2(r_1-2)}\}$, and thanks to \eqref{forp}, we know that  $p_{r_1}^{'}=\max\{2,\frac{4}{5(1-\frac{2}{r_1})}\}$, finally we obtain:
\begin{equation*}
\begin{aligned}
\|u_{\varepsilon,1}\|_{L^{{p_{r_1}^{'}}} L^{r_1}}\leq \|u_{\varepsilon,1}\|_{L^{{p_{r_1}^{'}}}  \dot{B}^0_{r_1,2}}\leq\|u_{\varepsilon,1}\|_{\widetilde{L}^{{p_{r_1}^{'}}} \dot{B}^0_{r_1,2}}\leq C \varepsilon^{(\frac{\theta_{r_1}^{'}}{2}-\frac{3\alpha}{4})(1-\frac{2}{r_1})}\|u_0\|_{L^2}.
\end{aligned}
\end{equation*}
\end{proof}

\begin{prop}\label{range223}
For any $2<r_2<3$, if we define $\theta_{r_2}^{''}=\min\{1,\frac{3-r_2}{r_2-2}\}$ and $p_{r_2}^{''}=\max\{2,\frac{2}{3-\frac{5}{r_2}}\}$ then
\begin{equation*}
    \begin{aligned}
\|u_{\varepsilon,2}\|_{L^{p_{r_2}^{''}}  L^{r_2}}\leq C \varepsilon^{\frac{\theta_{r_2}^{''}}{2}(1 - \frac{2}{r_2})-\alpha(2-\frac{3}{2r_2})}(\|u_0\|^2_{{L}^2}+\|b_0\|^2_{{L}^2}).
    \end{aligned}
\end{equation*}
\end{prop}
\begin{proof}
Combining Propositions \ref{ranger1} and \ref{range26}, the proof is straightforward and therefore omitted.
\end{proof}
\begin{theo}[Global in time estimate]\label{Global in time}
For any $r> 2$, there exists a constant $C=C(\|u_0\|_{L^2}, \|b_0\|_{L^2})$ such that we have the following global-in-time estimates:
\begin{equation*}
\|u_{\varepsilon}\|_{\widetilde{L}^{\max(1,\frac{4}{5(1-\frac{2}{r})})}\dot{B}^0_{r,2}+\widetilde{L}^{\max(1,\frac{2}{3-\frac{5}{r}})}\dot{B}^0_{r,2}}\leq C\varepsilon^{C(\alpha)},
\end{equation*}
where $$\alpha <\min \{\frac23\min(1,\frac{6-r}{2(r-2)}),\frac{r-2}{2(2r-3)}\min(1,\frac{3-r}{r-2})\},$$
and $$C(\alpha)=\min\{\frac12(1-\frac2r)-\alpha(2-\frac{3}{2r}),\frac{3-r}{2r}-\alpha(2-\frac{3}{2r})\}.$$

\end{theo}
\begin{proof}
Applying Propositions \ref{ranger1} and \ref{rangr2}, we obtain
\begin{equation*}\begin{aligned}
&\|u_{\varepsilon}\|_{\widetilde{L}^{\max(1,\frac{4}{5(1-\frac{2}{r})})}\dot{B}^0_{r,2}+\widetilde{L}^{\max(1,\frac{2}{3-\frac{5}{r}})}\dot{B}^0_{r,2}}\\=&
\inf_{u_\varepsilon=u_{\varepsilon,1}+u_{\varepsilon,2}}
\left(
\|u_{\varepsilon,1}\|_{\widetilde{L}^{\max(1,\frac{4}{5(1-\frac{2}{r})})}\dot{B}^0_{r,2}}+\|u_{\varepsilon,2}\|_{\widetilde{L}^{\max(1,\frac{2}{3-\frac{5}{r}})}\dot{B}^0_{r,2}}\right)\\\leq&
\inf_{u_\varepsilon=u_{\varepsilon,1}+u_{\varepsilon,2}}\Big(C\varepsilon^{\left(\frac{\theta_{r}}{2}-\frac{3\alpha}{4}\right)\left(1-\frac{2}{r}\right)}\|u_0\|^2_{{L}^2}+C \varepsilon^{\frac{\theta_{r}}{2}(1 - \frac{2}{r})-\alpha(2-\frac{3}{2r})}(\|u_0\|^2_{{L}^2}+\|b_0\|^2_{{L}^2})\Big)
\\\leq& C\varepsilon^{C(\alpha)}.
\end{aligned}
\end{equation*}
It suffices to obtain the smallest positive power of $\varepsilon$ ,  that means $$\left(\frac{\theta_r}{2}-\frac{3\alpha}{4}\right)\left(1-\frac{2}{r}\right)>0,\ \frac{\theta_{r}}{2}(1 - \frac{2}{r})-\alpha(2-\frac{3}{2r_2})>0,$$ so we obtain $$\alpha <\min \{\frac23\min(1,\frac{6-r}{2(r-2)}),\frac{r-2}{2(2r-3)}\min(1,\frac{3-r}{r-2})\}.$$
\end{proof}
\begin{rema}
    If we want positive powers of $\varepsilon$, we will ask a stronger condition on $\alpha$ [see next section 3.1.1].
\end{rema}

We calculate the local-in-time norm $\|u_\varepsilon\|_{L^2_tL^{r}}$ directly, for any $r\in (2,3)$:
\begin{equation}\label{L2r}
    \begin{aligned}
        \|u_\varepsilon\|_{L^2_tL^{r}} &=\|u_{\varepsilon,1}+u_{\varepsilon,2}\|_{L^2_t  L^{r}}\leq\|u_{\varepsilon,1}\|_{L^2_t  L^{r}}+\|u_{\varepsilon,2}\|_{L^2_t  L^{r}}.
\end{aligned}
\end{equation}
We apply Proposition \ref{range26} for $u_{\varepsilon,1}$, when ${r}\in(2,3),\ \theta_{r_1}^{'}$ is equal to 1 and $p_{r_1}^{'} $ is equal to $\frac{4}{5(1-\frac{2}{r})}$, so we obtain
\begin{equation}\label{u1}
\|u_{\varepsilon,1}\|_{L_t^{\frac{4}{5(1-\frac{2}{r})}}L^r}\leq C\varepsilon^{\left(\frac{1}{2}-\frac{3\alpha}{4}\right)\left(1-\frac{2}{r}\right)}
\|u_0\|_{L^2},
\end{equation}
and applying Proposition \ref{range223} for $u_{\varepsilon,2}$, when ${r}\in(2,\frac52],\ \theta_{r_2}^{'}$ is equal to 1 and $p_{r_2}^{'}$ is equal to
$\frac{2}{3-\frac{5}{r}}$, so we obtain
\begin{equation}\label{u21}
    \|u_{\varepsilon,2}\|_{L_t^{\frac{2}{3-\frac{5}{r}}}L^r}\leq \varepsilon^{\frac12\left(1-\frac{2}{r}\right)-\alpha\left(2-\frac{3}{2r}\right)}\big(\|u_0\|_{L^2}^2+\|b_0\|_{L^2}^2\big),
\end{equation}
when ${r}\in(\frac52,3),\theta_{r_2}^{'}$ is equal to $\frac{3-r}{r-2}$ and $p_{r_2}^{'} $ is equal to 2, that means
\begin{equation}\label{u22}
\|u_{\varepsilon,2}\|_{L_t^{2}L^r}\leq \varepsilon^{\frac{3-r}{2r}-\alpha\left(2-\frac{3}{2r}\right)}\big(\|u_0\|_{L^2}^2+\|b_0\|_{L^2}^2\big).
\end{equation}
Combining \eqref{u21} and \eqref{u22}, we can rewrite 
\begin{equation}\label{u2}
\|u_{\varepsilon,2}\|_{L_t^{\max(2,\frac{2}{3-\frac{5}{r}})}L^r}\leq \varepsilon^{\frac{1}{2}\cdot\min\{1,\frac{3-r}{r-2}\}\left(1-\frac{2}{r}\right)-\alpha\left(2-\frac{3}{2r}\right)}\big(\|u_0\|_{L^2}^2+\|b_0\|_{L^2}^2\big).
\end{equation}
Contuinuing calculation for any $t>0$, from \eqref{u1} and \eqref{u2}, we can directly obtain
\begin{equation}\label{u1t}
\|u_{\varepsilon,1}\|_{L^2_t  L^{r}}\leq
C t^{\max \{0, \frac{5}{2r}-\frac34\}}\varepsilon^{\left(\frac{1}{2}-\frac{3\alpha}{4}\right)\left(1-\frac{2}{r}\right)}
\|u_0\|_{L^2},
\end{equation}
and
\begin{equation}\label{u2t}
\|u_{\varepsilon,2}\|_{L^2_t  L^{r}}\leq C t^{\max \{0,\frac{5}{2r}-1\}}\varepsilon^{\frac12\min\{1,\frac{3-r}{r-2}\}\left(1-\frac{2}{r}\right)-\alpha\left(2-\frac{3}{2r}\right)}\big(\|u_0\|_{L^2}^2+\|b_0\|_{L^2}^2\big).
\end{equation}

Combining \eqref{u1t} and  \eqref{u2t}, the same aim is to obtain the smallest power of $\varepsilon$ positive, so we have
\begin{equation*}
\begin{aligned}
    \alpha&< \min\left\{\frac{2}{3},\frac{1}{2}\min\{\frac{3-r}{r-2},1\}\frac{1-\frac{2}{r}}{2-\frac{3}{2r}}\right\}\\&=\min\left\{\frac{2}{3},\frac{1}{4r-3}\min\{3-r,r-2\}\right\} \overset{def}{=} f(r).
\end{aligned}
\end{equation*}
Studying $f(r)$, we obtain that
\begin{equation*}
      f(r)=\begin{cases}
      \vspace{0.2cm}
          {\frac{r-2}{4r-3}}\quad \mbox{when}\quad 2<r\leq\frac{5}{2},\\
            {\frac{3-r}{4r-3}}\quad \mbox{when}\quad\frac{5}{2}< r< 3,
           \end{cases}
\end{equation*}
and $f$ is increasing in $(2,\frac52]$ and decreasing in $(\frac52,3)$ so that the maximal value for $\alpha$
is $\frac{1}{14}$.

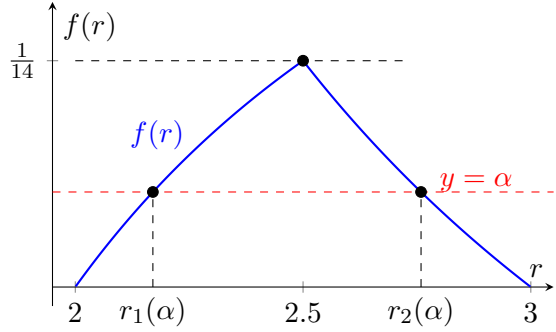
\begin{figure}[htbp]
\centering
\begin{tikzpicture}
\begin{axis}[
    width=8.2cm,
    height=5.6cm,
    axis lines=middle,
    xmin=1.95, xmax=3.05,
    ymin=-0.006, ymax=0.090,
    xtick={2,2.5,3},
    ytick={0,1/14},
    yticklabels={$0$,$\frac{1}{14}$},
    xlabel={$r$},
    ylabel={$f(r)$},
    samples=200,
    clip=false,
    scaled y ticks=false,
    enlargelimits=false
]

% choose alpha < 1/14
\def\alphaval{0.03}

% intersection formulas
\pgfmathsetmacro{\rone}{(2-3*\alphaval)/(1-4*\alphaval)}
\pgfmathsetmacro{\rtwo}{(3+3*\alphaval)/(1+4*\alphaval)}

% left branch
\addplot[thick, blue, domain=2.001:2.5] {(x-2)/(4*x-3)};

% right branch
\addplot[thick, blue, domain=2.5:2.999] {(3-x)/(4*x-3)};

% horizontal line y = alpha
\addplot[dashed, red, domain=1.95:3.05] {\alphaval};

% horizontal dashed line y = 1/14, starting from the y-axis
\addplot[dashed, black, domain=2:2.72] {1/14};

% vertical dashed lines from intersections to x-axis
\addplot[dashed, black] coordinates {(\rone,0) (\rone,\alphaval)};
\addplot[dashed, black] coordinates {(\rtwo,0) (\rtwo,\alphaval)};

% maximum point
\addplot[only marks, mark=*] coordinates {(2.5,1/14)};

% intersection points
\addplot[only marks, mark=*] coordinates {(\rone,\alphaval)};
\addplot[only marks, mark=*] coordinates {(\rtwo,\alphaval)};

% x-axis labels
\node[below] at (axis cs:\rone,0) {$r_1(\alpha)$};
\node[below] at (axis cs:\rtwo,0) {$r_2(\alpha)$};

% labels
\node[blue] at (axis cs:2.18,0.048) {$f(r)$};
\node[red] at (axis cs:2.88,\alphaval+0.003) {$y=\alpha$};
\end{axis}
\end{tikzpicture}
\end{figure}And once we fix $\alpha\in[0,\frac{1}{14}]$, there are two solutions to the equation $f(r)=\alpha$: $$r_1(\alpha)=\frac{3\alpha-2}{4\alpha-1},\quad r_2(\alpha)=\frac{3\alpha+3}{4\alpha+1}.$$

Then observing \eqref{u1t} and  \eqref{u2t}, we define 
\begin{equation}\label{Ar}
A_r=\max \{\frac{5}{2r}-\frac{3}{4},\frac{5}{2r}-1\},
\end{equation}
and
\begin{equation}\label{Br}
B_r=\min\{\frac{r-2}{2r},\frac{3-r}{2r}\}-\alpha(2-\frac{3}{2r}),
\end{equation}
we conclude for any $r\in (2,r_1(\alpha))\cup(r_2(\alpha),3)$
\begin{equation*}
     \|u_\varepsilon\|_{L^2_tL^{r}} \leq C \max\{1,t\}^{A_{r}} \,\varepsilon^{B_{r}}.
\end{equation*}

\subsection{For magnetic part}
First, recall that, thanks to the Duhamel formula and divergence theorem:
\begin{equation}\label{Duhamel}
    \begin{aligned}
     b_\varepsilon-e^{t\Delta}b_0&= -\int_0^{t}e^{(t-\tau)\Delta}(u_\varepsilon\cdot\nabla b_\varepsilon-b_\varepsilon\cdot\nabla u_\varepsilon) d\tau\\
     &=-\int_0^{t}e^{(t-\tau)\Delta}(\operatorname{div}(u_{\varepsilon}\otimes b_\varepsilon)-\operatorname{div}(b_{\varepsilon}\otimes u_\varepsilon)) d\tau\\
     &=-\int_0^{t}e^{(t-\tau)\Delta} \sum_{i}\partial_{i}\big( u_\varepsilon^i b_\varepsilon - b_\varepsilon^i u_\varepsilon \big).
\end{aligned}
\end{equation}
Then if we denote $K_t$ as the classical $3D$ heat kernel, that is
\begin{equation*}
    K_t=(4\pi t)^{-\frac32}e^{-\frac{|x|^2}{4t}}=t^{-\frac32}K_1(t^{-\frac12 }x),
\end{equation*}
we have 
\begin{equation*}
    e^{t\Delta}g=K_t*g.
\end{equation*}
Subsequently we can write that
\begin{equation*}
\begin{aligned}
b_\varepsilon-e^{t\Delta}b_0&= -\int_0^{t} K_{t-\tau}* \sum_{i=1}^3\partial_{i}\big( u_\varepsilon^i b_\varepsilon - b_\varepsilon^i u_\varepsilon \big)  \, d\tau \\& = -\sum_{i=1}^3\int_0^{t} \partial_{i}K_{t-\tau} *\big( u_\varepsilon^i b_\varepsilon - b_\varepsilon^i u_\varepsilon \big)  \, d\tau,
\end{aligned}
\end{equation*}
hence, the formula can be reduced to the simplified form:
\begin{equation*}
b_\varepsilon-e^{t\Delta}b_0= -\int_0^t (\nabla K_{t-\tau}) *
(u_{\varepsilon}\otimes b_\varepsilon -b_{\varepsilon}\otimes u_\varepsilon) d\tau.
\end{equation*}
\begin{rema}
    Making the derivative pound on the kernel is an argument also used in \cite{FC16,GP02,HO26}. 
\end{rema}
Using the H\"older estimates both in space and in time we have
\begin{equation}
\begin{aligned}\label{bvar}
   \|(b_\varepsilon-e^{t\Delta}b_0)\|_{L^{p}}&\leq \int_0^t \|\nabla K_{t-\tau}\|_{L^1}\|u_{\varepsilon}\otimes b_\varepsilon -b_{\varepsilon}\otimes u_\varepsilon\|_{L^{p}}  d\tau\\& \leq \int_0^t (t-\tau)^{-\frac12} \|u_{\varepsilon}\otimes b_\varepsilon -b_{\varepsilon}\otimes u_\varepsilon\|_{L^{p}}d\tau\\&\leq 2\big(\int_0^t (t-\tau)^{-\frac a2}d\tau\big)^{\frac{1}{a}}
\|u_{\varepsilon}\|_{L^{b}L^{r}}\|b_\varepsilon\|_{L^cL^q},
\end{aligned}
\end{equation}
where $$\frac{1}{a}+\frac{1}{b}+\frac{1}{c}=1,\ \ \ \ \frac{1}{p}=\frac{1}{r}+\frac{1}{q}.$$
The first integral in \eqref{bvar} exists if and only if $a<2.$ This implies
\begin{equation}\label{abc}
  \frac1b+\frac1c=1-\frac1a< \frac12\quad \Longrightarrow\quad  b,c>2.
\end{equation}

For $u_\varepsilon,$ we come back to \eqref{u1}, \eqref{u21} and \eqref{u22}: in \eqref{u22} $b=2$, which is not compatible with \eqref{abc}. So we only deal with the range $r\in(2,r_1(\alpha)),$ and in this situation, $$2<\frac{2}{3-\frac{5}{r}}\leq \frac{4}{5(1-\frac2r)}$$ always holds, so $$b=\frac{2}{3-\frac{5}{r}}, \ \frac1b\in(\frac14,\frac{1+11\alpha}{4-6\alpha}).$$
From \eqref{u1} and \eqref{u21}, using the elementary fact that if $p'\leq p$ we have $\|g\|_{L_t^{p'}}\leq t^{\frac1{p'}-\frac1{p}} \|g\|_{L_t^p}$ we obtain:
\begin{equation*}
\begin{aligned}
\|u_{\varepsilon}\|_{L_t^{\frac{2}{3-\frac{5}{r}}}L^r}&\leq t^{\frac14}\|u_{\varepsilon,1}\|_{L_t^{\frac{4}{5(1-\frac{2}{r})}}L^r}+\|u_{\varepsilon,2}\|_{L_t^{\frac{2}{3-\frac{5}{r}}}L^r} \\
&\leq C (1+t^{\frac14})\varepsilon^{\frac12\left(1-\frac{2}{r}\right)-\alpha\left(2-\frac{3}{2r}\right)}.
\end{aligned}
\end{equation*}

For $b_\varepsilon$, we interpolate between the energy bound $b_\varepsilon\in L_t^{\infty}L^2$ and the parabolic smoothing $b_\varepsilon\in L_t^2L^6$, and for any $q\in[2,6],$  we obtain 
\begin{equation*}
    \|b_\varepsilon\|_{L^cL^q}\leq \|b_\varepsilon\|_{L^\infty L^2}^{1-\theta}\|b_\varepsilon\|_{L^2L^6}^{\theta}\leq C(\|u_0\|_{L^2}+\|b_0\|_{L^2})^{\frac12},
\end{equation*}
 where $c=\frac{4q}{3q-6}.$ And $c$ satisfies \eqref{abc} if $q\in[2,6)$ $$\frac1c=\frac34-\frac{3}{2q}< \frac12.$$
Now we come back to \eqref{bvar}, we take $b=\frac{2}{3-\frac{5}{r}}$ and $c=\frac{4q}{3q-6}$,
so we can obtain
\begin{equation*}
    \frac1a=\frac{5}{2r}+\frac{3}{2q}-\frac54,
\end{equation*}
and 
\begin{equation}\label{aless2}
    a<2\Longleftrightarrow \frac{5}{2r}+\frac{3}{2q}-\frac74>0.
\end{equation}
Because our $r\in (2,r_1(\alpha))$, \eqref{aless2} implies that $q\in[2,\frac{6(2-3\alpha)}{4+19\alpha}).$ And obviously we can obtain $p\in(1,\frac{6(2-3\alpha)}{5(2-\alpha)})$ from \eqref{abc}.

Finally, we obtain
\begin{equation*}
\begin{aligned}
    \|(b_\varepsilon-e^{t\Delta}b_0)\|_{L^{p}}
&\leq(\frac{2}{2-a}t^{1-\frac a2})^{\frac1a}(1+t^{\frac14})\varepsilon^{\frac12(1-\frac2r)-\alpha(2-\frac{3}{2r})}C(\|u_0\|_{L^2},\|b_0\|_{L^2})\\&\leq C\max\{1,t\}^{M_{p,r}}\varepsilon^{N_r}.
\end{aligned}
\end{equation*}
where $$M_{p,r}=\frac{5}{2r}-\frac32+\frac{3}{2p},\,\,N_r=\frac{r-2}{2r}-\alpha(2-\frac{3}{2r})$$ and the constant C depends on $\|u_0\|_{L^2}^2,\|b_0\|_{L^2}^2,r,p$.
\begin{rema}
When $r=\frac52$, $\alpha$ can touch $\frac{1}{14},$ that means $ N_r=0$ which only provides a finite bound instead of a convergence rate.
\end{rema}
This completes the proof for the magnetic part. And from \eqref{bvar} we finally obtained all the bounds for the norms of $u_\varepsilon$ and $b_\varepsilon$.
\subsection{Parameter relations}
Thus we need to clarify the relation between $p$ and $r$. Actually for a given $p$, obtaining the admissible $r$ requires to cut the range of $p$ in several parts. More precisely, as we outlined above, when $r\in(2,r_1(\alpha))$ and $q\in[2,\frac{6(2-3\alpha)}{4+19\alpha})$, then $p\in (1,\frac{6(2-3\alpha)}{5(2-\alpha)})$. And for a given $p\in (1,\frac{6(2-3\alpha)}{5(2-\alpha)})$ the question is to specify the following set:
$$
\mathcal{I}_p= \{r\in(2,r_1(\alpha))|\;\exists q\in[2,\frac{6(2-3\alpha)}{4+19\alpha})\text{ such that }\frac1{p}=\frac1{r}+\frac1{q}\},$$
which is the object of the following proposition:
\begin{prop}
 \sl{Under the previous notations, if $\alpha\in(0,\frac2{43}]$ we have
 \begin{itemize}
\item $p\in\left(1,\frac{3(2-3\alpha)}{5(\alpha+1)}\right)
\Longrightarrow \mathcal{I}_p= \left(2,\frac{2p}{2-p}\right),$
\item $p\in\left(\frac{3(2-3\alpha)}{5(\alpha+1)},\frac{2(2-3\alpha)}{4-11\alpha}\right)\Longrightarrow \mathcal{I}_p= \left(\frac{2p}{2-p},\frac{6p(2-3\alpha)}{12-18\alpha-4p-19p\alpha}\right),$
\item $p\in\left(\frac{2(2-3\alpha)}{4-11\alpha},\frac{6(2-3\alpha)}{5(2-\alpha)}\right)\Longrightarrow
\mathcal{I}_p= \left(\frac{6p(2-3\alpha)}{12-18\alpha-4p-19p\alpha},r_1(\alpha)\right).$
\end{itemize}
 If $\alpha\in(\frac2{43},\frac1{14})$, we have
 \begin{itemize}
\item $p\in\left(1,\frac{2(2-3\alpha)}{4-11\alpha}\right)
\Longrightarrow \mathcal{I}_p= \left(2,\frac{2p}{2-p}\right),$
\item $p\in\left(\frac{2(2-3\alpha)}{4-11\alpha},\frac{3(2-3\alpha)}{5(\alpha+1)}\right)\Longrightarrow \mathcal{I}_p= \left(2,r_1(\alpha)\right),$
\item $p\in\left(\frac{3(2-3\alpha)}{5(\alpha+1)},\frac{6(2-3\alpha)}{5(2-\alpha)}\right)\Longrightarrow\mathcal{I}_p= \left(\frac{6p(2-3\alpha)}{12-18\alpha-4p-19p\alpha},r_1(\alpha)\right).$
\end{itemize}
 }
\end{prop}
\begin{proof}
We denote that
\begin{itemize}
    \item $R=\frac{1}{r_1(\alpha)}\in(\frac{1-4\alpha}{2-3\alpha},\frac12):=(m_1,m_2),$
    \item $Q=\frac1q\in(\frac{4+19\alpha}{6(2-3\alpha)},\frac12):=(q_1,q_2),$
    \item $P=\frac1p\in(\frac{5(2-\alpha)}{6(2-3\alpha)},1):=(p_1,p_2)$.
\end{itemize}
We want to explain $\mathcal{I}_{p}$ in more detail: first,
 we introduce four points
 \begin{equation}
      (A,B,C,D)=\left((m_1,q_2),(m_2,q_2),(m_1,q_1),(m_2,q_1)\right).
 \end{equation}
 And we have 
\begin{equation}
\begin{aligned}
    (m_1+q_2,q_2+m_2,m_1+q_1,m_2+q_1)
        =\left(\frac{4-11\alpha}{2(2-3\alpha)},1,\frac{5(2-\alpha)}{6(2-3\alpha)},\frac{5(1+\alpha)}{3(2-3\alpha)}\right).
    \end{aligned}
\end{equation}
Fix \(p\in(1,\frac{6(2-3\alpha)}{5(2-\alpha)})\), then the identity $\frac1p=\frac1r+\frac1q$
becomes $P=R+Q,$ that is, a straight line in the \((R,Q)\)-plane with slope \(-1\). Meanwhile, the admissibility conditions on \(r\) and \(q\) imply $(R,Q)$ must belong to the rectangle $(m_1,m_2)\times(q_1,q_2).$
Hence the admissible set \(\mathcal I_p\) is precisely the projection onto the R-axis of the intersection of the line \(R+Q=P\) with this rectangle. This is illustrated in the following two pictures.\\
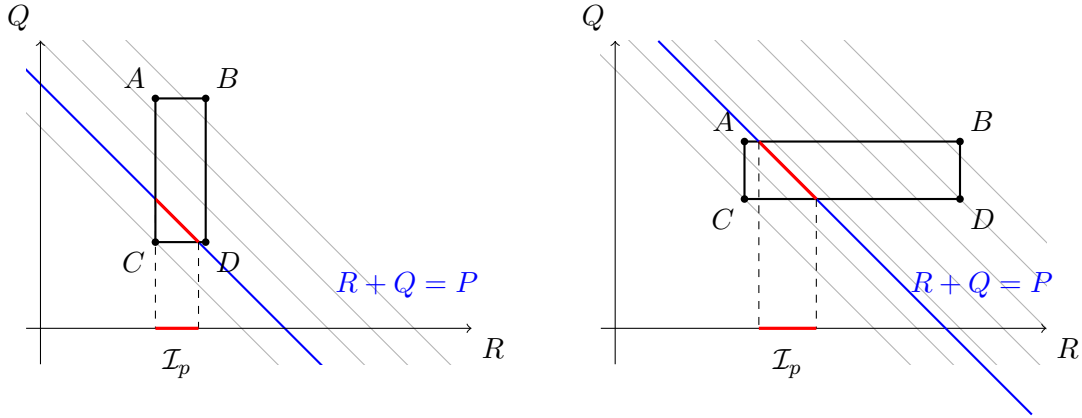
\begin{center}
\begin{tikzpicture}[scale=0.95]
\def\xmin{-0.2}
\def\xmax{6}
\def\ymin{-0.5}
\def\ymax{4}

%==================== Left figure ====================
\begin{scope}
  % rectangle corners
  \coordinate (C) at (1.6,1.2); % lower-left
  \coordinate (D) at (2.3,1.2); % lower-right
  \coordinate (A) at (1.6,3.2); % upper-left
  \coordinate (B) at (2.3,3.2); % upper-right

  % background diagonal family
  \begin{scope}
    \clip (\xmin,\ymin) rectangle (\xmax,\ymax);
    % 灰线（去掉 3.4）
    \foreach \c in {2.8,4.0,4.6,5.2} {
      \draw[gray!60] (-20,\c+20) -- (20,\c-20);
    }
    % 第二条改为蓝线（R+Q=3.4）
    \draw[blue, thick] (-20,3.4+20) -- (20,3.4-20);
  \end{scope}

  % axes
  \draw[->] (\xmin,0) -- (\xmax,0) node[below right] {$R$};
  \draw[->] (0,\ymin) -- (0,\ymax) node[above left] {$Q$};

  % rectangle
  \draw[thick] (C) rectangle (B);

  % corner labels
  \fill (A) circle (1.5pt) node[above left] {$A$};
  \fill (B) circle (1.5pt) node[above right] {$B$};
  \fill (C) circle (1.5pt) node[below left] {$C$};
  \fill (D) circle (1.5pt) node[below right] {$D$};

  % label
  \node[blue] at (5.1,0.6) {$R+Q=P$};

  % 红色交集（R+Q=3.4 与矩形）
  % 左边界 R=1.6 → Q=1.8
  % 下边界 Q=1.2 → R=2.2
  \coordinate (L1) at (1.6,1.8);
  \coordinate (L2) at (2.2,1.2);
  \draw[red, very thick] (L1) -- (L2);

  % projections
  \draw[dashed] (L1) -- (1.6,0);
  \draw[dashed] (L2) -- (2.2,0);
  \draw[red, very thick] (1.6,0) -- (2.2,0);
  \node[below] at (1.9,-0.15) {$\mathcal I_p$};

\end{scope}

%==================== Right figure ====================
\begin{scope}[xshift=8cm]
  % rectangle corners
  \coordinate (C) at (1.8,1.8); % lower-left
  \coordinate (D) at (4.8,1.8); % lower-right
  \coordinate (A) at (1.8,2.6); % upper-left
  \coordinate (B) at (4.8,2.6); % upper-right

  % background diagonal family
  \begin{scope}
    \clip (\xmin,\ymin) rectangle (\xmax,\ymax);
    \foreach \c in {3.6,4.2,4.8,5.4,6.0,6.6,7.2} {
      \draw[gray!60] (-20,\c+20) -- (20,\c-20);
    }
  \end{scope}

  % axes
  \draw[->] (\xmin,0) -- (\xmax,0) node[below right] {$R$};
  \draw[->] (0,\ymin) -- (0,\ymax) node[above left] {$Q$};

  % rectangle
  \draw[thick] (C) rectangle (B);

  % corner labels
  \fill (A) circle (1.5pt) node[above left] {$A$};
  \fill (B) circle (1.5pt) node[above right] {$B$};
  \fill (C) circle (1.5pt) node[below left] {$C$};
  \fill (D) circle (1.5pt) node[below right] {$D$};

  % one example line
  \draw[blue, thick] (0.6,4.0) -- (5.8,-1.2);
  \node[blue] at (5.1,0.6) {$R+Q=P$};

  % 红色交集（保持你原来的）
  \coordinate (R1) at (2.0,2.6);
  \coordinate (R2) at (2.8,1.8);
  \draw[red, very thick] (R1) -- (R2);

  % projections
  \draw[dashed] (R1) -- (2.0,0);
  \draw[dashed] (R2) -- (2.8,0);
  \draw[red, very thick] (2.0,0) -- (2.8,0);
  \node[below] at (2.4,-0.15) {$\mathcal I_p$};

\end{scope}

\end{tikzpicture}
\end{center}
To rigourously prove the previous proposition we need to specify for $P\in(a+b,1)$, the following set

$$\mathcal{J}_P=\{R\in (a,\frac12)|\; \exists Q\in(b,\frac12)\ \ \text{such that}\ \  P=R+Q\},$$
which is done in the following Lemma \ref{abP} (and the proposition will be proved taking  $a=m_1,b=p_1$).
\end{proof}
\begin{lemm}\label{abP}
For any $a,b \in (0,\frac12),$ for any $P\in(a+b,1)$:

Case 1: When $a<b$
\begin{itemize}
    \item  $P\in (a+b,a+\frac12) \Longrightarrow \mathcal{J}_p=[a,P-b)$,
    \item  $P\in [\frac12+a,\frac12+b] \Longrightarrow \mathcal{J}_p=[P-\frac12,P-b)$,
    \item  $P\in [\frac12+b,1) \Longrightarrow \mathcal{J}_p=[P-\frac12,\frac12)$,
\end{itemize}
Case 2: When $a\geq b$
\begin{itemize}
    \item  $P\in (a+b,a+\frac12+b) \Longrightarrow \mathcal{J}_p=[a,P-b)$,
    \item  $P\in [\frac12+a,\frac12+b] \Longrightarrow \mathcal{J}_p=[a,\frac12)$,
    \item  $P\in [\frac12+a,1) \Longrightarrow \mathcal{J}_p=[P-\frac12,\frac12)$.
\end{itemize}
\end{lemm}

Combining the three parts, we obtain Theorem \ref{mainthe1}.
\section{Strong solutions}
\subsection{Precise statement of our result}
Let us begin with stating more precisely our main results concerning strong solutions: first we already have local strong solution on $[0,T_\varepsilon^*)$ (\cite{BIM05,FCVSN25}), then we give the global existence and convergence theorem for very large ill-prepared initial data in this article.

The lack of dispersion for the magnetic field in \eqref{MHDalpha} prevents us from directly applying Strichartz estimates. We isolate the purely diffusive part of magnetic field. Here the idea comes from the limit system in \cite{FCVSN25}, in the special case $\tilde{u}=\tilde{b}=0:$
\begin{equation}\label{cvar}
\left\{
\begin{aligned}
&\partial_t c_\varepsilon - \Delta c_\varepsilon
= 0, \\
&c_\varepsilon\big|_{t=0} = b_{0,\varepsilon},
\end{aligned}
\right.
\end{equation}
then for any $t\ge0$ we have that
\begin{equation}\label{enC}
\|c_\varepsilon(t)\|_{L^2}^2
+ 2\varepsilon^\alpha \int_0^t \|\nabla c_\varepsilon(s)\|_{L^2}^2\,ds
= \|b_{0,\varepsilon}\|_{L^2}^2.
\end{equation}
Actually, when $u_\varepsilon$ is very small, the terms of the second equation in \eqref{MHDalpha} $u_\varepsilon\cdot\nabla b_\varepsilon$ and $b_\varepsilon\cdot\nabla u_\varepsilon$ are also very small. Then the solution $c_\varepsilon=e^{t\Delta}b_{0,\varepsilon}$ of \eqref{cvar} is the main part of magnetic part, in the next following part we will show that $b_\varepsilon-c_\varepsilon\longrightarrow 0.$

Therefore, we define the velocity driven only by the linear rotation
operator and the forcing generated by $c_\varepsilon$: (since $c_\varepsilon$ has no reason to be equal to 0, we leave the inhomogeneous term as it is)
\begin{equation}\label{Wvar}
\left\{
\begin{aligned}
&\partial_t W_\varepsilon - \varepsilon^\alpha \Delta W_\varepsilon
+ \frac{1}{\varepsilon} \mathbb{P}(W_\varepsilon \wedge e_3)
= \mathbb{P}\big(c_\varepsilon \cdot \nabla c_\varepsilon \big),\\
&W_\varepsilon\big|_{t=0} = u_{0,\varepsilon}.
\end{aligned}
\right.
\end{equation}
Then we have isolated the linear part \eqref{cvar} and \eqref{Wvar} in \eqref{MHDalpha}, and we can estimate these two parts by using energy equality for \eqref{cvar} in \cite{FCVSN25} and Strichartz estimates for \eqref{Wvar}.
If we set $D_\varepsilon=(\delta_\varepsilon,d_\varepsilon)=(u_\varepsilon-W_\varepsilon,b_\varepsilon-c_\varepsilon)$ then it satisfies the following system:
\begin{equation}\label{Dvar}
\left\{
\begin{aligned}
&\partial_t \delta_\varepsilon
- \varepsilon^\alpha \Delta \delta_\varepsilon
+ \frac{1}{\varepsilon}\mathbb{P}(\delta_\varepsilon \wedge e_3)
= \sum_{i=1}^{7} F_i, \\[0.6em]
&\partial_t d_\varepsilon
-  \Delta d_\varepsilon
= \sum_{i=1}^{8} G_i, \\[0.6em]
&(\delta_\varepsilon, d_\varepsilon)\big|_{t=0} = (0, 0),
\end{aligned}
\right.
\end{equation}
where
\begin{equation*}
\left\{
\begin{aligned}
F_1&= -\mathbb{P}(\delta_\varepsilon \cdot \nabla \delta_\varepsilon), \quad
F_2= -\mathbb{P}(\delta_\varepsilon \cdot \nabla W_\varepsilon), \quad
F_3= -\mathbb{P}(W_\varepsilon \cdot \nabla \delta_\varepsilon),\quad 
F_4= -\mathbb{P}(W_\varepsilon \cdot \nabla W_\varepsilon), \\
F_5&=\mathbb{P}(d_\varepsilon \cdot \nabla d_\varepsilon), \quad 
F_6=\mathbb{P}(d_\varepsilon \cdot \nabla c_\varepsilon), \quad
F_7=\mathbb{P}(c_\varepsilon \cdot \nabla d_\varepsilon), \\
G_1&=-\delta_\varepsilon \cdot \nabla d_\varepsilon, \quad
G_2=-W_\varepsilon \cdot \nabla d_\varepsilon, \quad
G_3= d_\varepsilon \cdot \nabla \delta_\varepsilon, \quad
G_4= d_\varepsilon \cdot \nabla W_\varepsilon, \quad\\
G_5&= -\delta_\varepsilon \cdot \nabla c_\varepsilon, \quad
G_6=-W_\varepsilon \cdot \nabla c_\varepsilon, \quad
G_7= c_\varepsilon \cdot \nabla \delta_\varepsilon, \quad
G_8= c_\varepsilon \cdot \nabla W_\varepsilon.
\end{aligned}
\right.
\end{equation*}
Let us write here a more precise version of Theorem \ref{mainthe2}.
\begin{theo}[Strong solutions, precised version]\label{mainthe22}
For any $\delta\in(0,\frac13)$(extra regularity), $\gamma\in[0,\frac\delta2),\alpha\in[0,\frac{1}{8})$ such that $$\alpha(\frac32-\frac\delta2)+\gamma<\frac\delta2,$$for any $ C_0>0$, and any $k\in(0,1),$ there exist $\varepsilon_0,m_0,l\in(0,1]$, all three depending on  $\delta,\gamma,\alpha,C_0,k,$ such that for any $\varepsilon\leq \varepsilon_0,$ any initial data satisfying
\begin{equation} \label{stronginitial}
\|u_{0,\varepsilon}\|_{\dot H^{\tfrac12-\delta}\cap \dot H^{\tfrac12+\delta}}
\leq C_0\varepsilon^{-\gamma},\quad
\|b_{0,\varepsilon}\|_{H^{\tfrac12+\delta}}
\leq m_0 |\ln \varepsilon|^{\tfrac12},
\end{equation}
there exists a unique global strong solution of \eqref{Dvar}, such that for any $t\geq 0,$ and any $s \in[\frac12-l(\frac\delta2-\gamma),\frac12+l(\frac\delta2-\gamma)],$ we have 
\begin{equation*}
\|\delta_\varepsilon(t)\|_{\dot{H}^{s}}^2+\|d_\varepsilon(t)\|_{\dot{H}^{s}}^2+\varepsilon^\alpha \int_0^t\|\nabla \delta_\varepsilon(\tau)\|_{\dot{H}^{s}}^2 d\tau +\int_0^t\|\nabla d_\varepsilon(\tau)\|_{\dot{H}^{s}}^2 d\tau \leq \varepsilon^{2[k(\frac\delta2-\gamma)-\alpha(\frac12-\frac\delta2)]}.
\end{equation*}
Moreover
\begin{equation*}
\|(\varepsilon^{\frac\alpha2}u_\varepsilon,b_\varepsilon-e^{t\Delta}b_{0,\varepsilon})\|_{L^2L^{\infty}}\leq C_0\varepsilon^{k(\frac\delta2-\gamma)-\alpha(\frac12-\frac\delta2)}.
\end{equation*} 
\end{theo}
\begin{rema}
To make the following calculation easier, we introduce $\eta_0\in (0,\frac12],$ such that $\gamma=\frac\delta2(1-\delta\eta_0).$
\end{rema}
\subsection{Proof of Theorem \ref{mainthe22} }
We know from \cite{BIM05} [Theorem 1.3] that there exists a local solution defined on $[0,T_\varepsilon^*[$. Assume, by contradiction, that $T_\varepsilon^*<+\infty$; then, by the continuation criterion, we have:
\begin{equation*}
   \int_0^{T_\varepsilon^*} \|\nabla u_\varepsilon\|^2_{\dot{H}^{\frac12}}\,d \tau =+\infty.
\end{equation*}
We can now define the following time (where the universal constant $C$ is given below in the bounds involving $F_1,F_5$ and $G_1,G_3$):
\begin{equation}\label{Time}
 T_\varepsilon=\sup \{t\in[0,T_\varepsilon^*[, \; \forall t'\leq t, \|\delta_\varepsilon(t')\|_{\dot{H}^\frac12} +\|d_\varepsilon(t')\|_{\dot{H}^\frac12}\leq \frac1{4C}\min(\varepsilon^\alpha,1)\}.
\end{equation}
Here, we will present the computations for some $s\in[\frac12-l(\frac\delta2-\gamma),\frac12+l(\frac\delta2-\gamma)],$ but we will first choose $s=\frac12$ to show $T_\varepsilon=T_\varepsilon^*=+\infty$, and then state global estimate for other $s.$
As the initial data from System \eqref{Dvar} is zero, we know that $T_\varepsilon>0$. Let us now assume, by contradiction, that
\begin{equation*}
 T_\varepsilon<T_\varepsilon^*.
\end{equation*}
Now, performing inner products in $\dot{H}^s$ of System \eqref{Dvar} with $(\delta_\varepsilon,d_\varepsilon)$, we obtain that for all $t\leq T_\varepsilon$:
\begin{equation}\label{energy}
\begin{cases}
  \frac12 \frac{d}{dt}\|\delta_\varepsilon(t)\|_{\dot{H}^s}^2 +\varepsilon^{\alpha} \|\nabla \delta_\varepsilon(t)\|_{\dot{H}^s}^2 \leq \sum_{k=1}^{7} |(F_k(t)|\delta_\varepsilon(t))_{\dot{H}^s}|,\\[8pt]
  \frac12 \frac{d}{dt}\|d_\varepsilon(t)\|_{\dot{H}^s}^2 +\|\nabla d_\varepsilon(t)\|_{\dot{H}^s}^2 \leq \sum_{k=1}^{8} |(G_k(t)|d_\varepsilon(t))_{\dot{H}^s}|.
\end{cases}
\end{equation}

Most of the right-hand-side terms are treated as in the case of rotating fluids, or very similarly (see details in \cite{FCVSN25}). We will provide no details for these terms and more focus on the new difficult terms: $F_6,F_7,G_7,G_8$.\\
Let us begin with the easiest terms, thanks to Lemma \ref{Sobolevproductlaw}.
\begin{equation}\label{est14}
\begin{cases}
    \begin{aligned}
&|(F_1|\delta_\varepsilon)_{\dot{H}^s}| \leq C_s\|\delta_\varepsilon\|_{\dot{H}^\frac12} \|\nabla  \delta_\varepsilon\|_{\dot{H}^s}^2,\\
&|(F_2|\delta_\varepsilon)_{\dot{H}^s}| \leq \frac{\varepsilon^{\alpha}}{2025} \|\nabla \delta_\varepsilon\|_{\dot{H}^s}^2 +\frac{C_s}{\varepsilon^{\alpha}} \|\nabla W_\varepsilon\|_{L^3}^2 \|\delta_\varepsilon\|_{\dot{H}^s}^2,\\
&|(F_3|\delta_\varepsilon)_{\dot{H}^s}| \leq \frac{\varepsilon^{\alpha}}{2025} \|\nabla \delta_\varepsilon\|_{\dot{H}^s}^2 +\frac{C_s}{\varepsilon^{{3\alpha}}} \|W_\varepsilon\|_{L^6}^4 \|\delta_\varepsilon\|_{\dot{H}^s}^2,\\
&|(F_4|\delta_\varepsilon)_{\dot{H}^s}| \leq \frac{\varepsilon^{\alpha}}{2025} \|\nabla \delta_\varepsilon\|_{\dot{H}^s}^2 +\frac{C_s}{\varepsilon{^{\alpha\frac{s}{1-s}}}} \|W_\varepsilon\|_{L^6}^{\frac{2}{1-s} }\|\delta_\varepsilon\|_{\dot{H}^s}^2 +C\|\nabla W_\varepsilon\|_{L^3}^2,
\end{aligned}
\end{cases}
\end{equation}
and
\begin{equation}\label{est48}
\begin{cases}
\begin{aligned}
& |(G_1|d_\varepsilon)_{\dot{H}^s}| \leq C_s\|\delta_\varepsilon\|_{\dot{H}^\frac12} \|\nabla d_\varepsilon\|_{\dot{H}^s}^2,\\
&|(G_2|d_\varepsilon)_{\dot{H}^s}| \leq \frac{1}{2026} \|\nabla  d_\varepsilon\|_{\dot{H}^s}^2 +C_s \|W_\varepsilon\|_{L^6}^4 \|d_\varepsilon\|_{\dot{H}^s}^2,\\
 &|(F_5|\delta_\varepsilon)_{\dot{H}^s}| +|(G_3|d_\varepsilon)_{\dot{H}^s}| \leq C_s\|d_\varepsilon\|_{\dot{H}^\frac12} \left(\|\nabla  \delta_\varepsilon\|_{\dot{H}^s}^2 +\|\nabla  d_\varepsilon\|_{\dot{H}^s}^2\right),\\
&|(G_4|d_\varepsilon)_{\dot{H}^s}| \leq \frac{1}{2026} \|\nabla  d_\varepsilon\|_{\dot{H}^s}^2 +C_s\|\nabla  W_\varepsilon\|_{L^3}^2 \|d_\varepsilon\|_{\dot{H}^s}^2,\\
 &|(G_5|d_{\varepsilon})_{\dot{H}^s}|\leq \frac{1}{2026}\|\nabla d_\varepsilon\|^2_{\dot{H}^s}+C_s\|\nabla c_\varepsilon\|^2_{\dot{H}^{\frac{1}{2}}}\|\delta_\varepsilon\|^2_{\dot{H}^s},\\
 &|(G_6|d_{\varepsilon})_{\dot{H}^s}|\leq\frac{1}{2026}\|\nabla d_{\varepsilon}\|^2_{\dot{H}^{s}}+C_s\|\nabla c_\varepsilon\|_{\dot{H}^{\frac{1}{2}}}^2\| d_{\varepsilon}\|_{\dot{H}^{s}}^{2}+\|W_\varepsilon\|^2_{L^6}\|\nabla c_\varepsilon\|_{\dot{H}^{\frac{1}{2}}}^{2s}.
 \end{aligned}
\end{cases}
\end{equation}
Estimates \eqref{est14} and \eqref{est48} correspond to the particular case where the coefficient $\nu^{'}$ equals 1 in \cite{FCVSN25}. For the remaining terms, the proof relies on the use of nonlocal operators as in \cite{FC20,FCsharper}. Indeed, if we use the estimates from \cite{FCVSN25} for $F_6,F_7,G_7,G_8$, negative powers of $\nu=\varepsilon^{\alpha}$ now appear, and there is no possible use of the Strichartz estimate of $W_\varepsilon$ to balance it. 

In order to deal with these new difficulties, we first state the following proposition, for which we recall that $|D|^{s}=(-\Delta)^{\frac s2}$.

\begin{prop}[\cite{FCsharper}]\label{Nonlocal}
For any $s \in (0,1)$ and any smooth functions $f,g$ we can write
\begin{equation*}
|D|^{s}(fg) = (|D|^{s}f)g + f|D|^{s}g + M_s(f,g),
\end{equation*}
where the bilinear operator $M_s$ is defined for all $x \in \R^3$ by
\begin{equation*}
M_s(f,g)(x)
=
\int_{\R^3}
\frac{(f(x)-f(x-y))(g(x)-g(x-y))}{|y|^{3+s}}\,dy .
\end{equation*}

Moreover, there exists a constant $\mathbf{C}$ (here $\mathbf{C}$ depends on $s$) such that for all $f,g$ and all
$r,r_1,r_2,q_1,q_2 \in [1,\infty]$ and $s_1,s_2>0$ satisfying
\begin{equation*}
\frac{1}{r}=\frac{1}{r_1}+\frac{1}{r_2},
\qquad
1=\frac{1}{q_1}+\frac{1}{q_2},
\qquad
s_1+s_2=s,    
\end{equation*}
we have
\begin{equation*}
\|M_s(f,g)\|_{L^r}
\leq
\mathbf{C}
\|f\|_{\dot{B}^{s_1}_{r_1,q_1}}
\|g\|_{\dot{B}^{s_2}_{r_2,q_2}} .
\end{equation*}
\end{prop}

Next, we consider the term $F_6$, by H\"older inequality and embedding inequality, we obtain
\begin{equation*}
\begin{aligned}
&|(F_6|\delta_{\varepsilon})_{\dot{H}^s}|=|((d_{\varepsilon}\cdot \nabla  c_{\varepsilon})|\delta_{\varepsilon})_{\dot{H}^s}|\leq\||D|^{s}(d_{\varepsilon}\cdot \nabla  c_{\varepsilon})\|_{L^{2}} \||D|^{s}\delta_{\varepsilon}\|_{L^2} \\
\leq& C_s(\||D|^{s}d_{\varepsilon}\|_{L^2} \|\nabla  c_{\varepsilon}\|_{L^{\infty}} +\|d_{\varepsilon}\|_{L^{\frac{6}{3-2s}}}\||D|^{s} \nabla  c_{\varepsilon}\|_{L^{\frac{3}{s}}}+\|M_s(d_{\varepsilon}, \nabla c_{\varepsilon})\|_{L^2})\|\delta_{\varepsilon}\|_{\dot{H}^s}\\
\leq& C_s(\|d_{\varepsilon}\|_{\dot{H}^s}\|\nabla  c_{\varepsilon}\|_{\dot{B}^{0}_{\infty,1}}+\|d_{\varepsilon}\|_{\dot{H}^{s}}\||D|^{s} \nabla  c_{\varepsilon}\|_{\dot{H}^{\frac{3}{2}-s}}+\|M_s(d_{\varepsilon}, \nabla c_{\varepsilon})\|_{L^2})\|\delta_{\varepsilon}\|_{\dot{H}^s},
    \end{aligned}
\end{equation*}
Then we use Proposition \ref{Nonlocal}, with $r_1=\frac{2}{1-2\eta},r_2=\frac1\eta,q_1=q_2=2,s_1=s-\beta,s_2=\beta,$ for some $\beta\in(0,s)$ and $\eta\in(0,1)$, 
\begin{equation*}
\begin{aligned}
 |(F_6|\delta_{\varepsilon})_{\dot{H}^s}|\leq C_s(\|d_{\varepsilon}\|_{\dot{H}^s}\|\nabla  c_{\varepsilon}\|_{\dot{B}^{\frac{3}{2}}_{2,1}}+\|d_{\varepsilon}\|_{\dot{H}^{s}}\|\nabla  c_{\varepsilon}\|_{\dot{H}^{\frac{3}{2}}}  +\mathbf{C}\|d_{\varepsilon}\|_{\dot{B}^{s-\beta}_{\frac{2}{1-2\eta},2}}\|\nabla c_{\varepsilon}\|_{\dot{B}^{\beta}_{\frac{1}{\eta},2}})\|\delta_{\varepsilon}\|_{\dot{H}^s},
    \end{aligned}
\end{equation*}
following by Bernstein inequality 
\begin{equation*}
\begin{aligned}
 |(F_6|\delta_{\varepsilon})_{\dot{H}^s}|\leq
C_s(\|d_{\varepsilon}\|_{\dot{H}^s}\|\nabla  c_{\varepsilon}\|_{\dot{B}^{\frac{3}{2}}_{2,1}}+\|d_{\varepsilon}\|_{\dot{H}^{s}}\|\nabla  c_{\varepsilon}\|_{\dot{H}^{\frac{3}{2}}}  +\mathbf{C}\|d_{\varepsilon}\|_{\dot{H}^{s-\beta+3\eta}}\|\nabla c_{\varepsilon}\|_{\dot{H}^{\frac{3}{2}+\beta-3\eta}})\|\delta_{\varepsilon}\|_{\dot{H}^s}.
    \end{aligned}
\end{equation*}
Finally if we choose $\eta$ and $\beta$ satisfying $\beta=3\eta,$ then we have
\begin{equation}\label{estF6}
    |(F_6|\delta_{\varepsilon})_{\dot{H}^s}|\leq C_s\|d_{\varepsilon}\|_{\dot{H}^s}\|\nabla  c_{\varepsilon}\|_{\dot{B}^{\frac{3}{2}}_{2,1}}\|\delta_{\varepsilon}\|_{\dot{H}^s}\leq \frac {C_s}2\|\nabla  c_{\varepsilon}\|_{\dot{B}^{\frac{3}{2}}_{2,1}}(\|d_{\varepsilon}\|_{\dot{H}^s}^2+\|\delta_{\varepsilon}\|_{\dot{H}^s}^2).
\end{equation}
The process for the estimates with $F_7,G_7,$ are similar to $F_6$:
\begin{equation}\label{estF7}
\begin{aligned}
&|(F_7|\delta_{\varepsilon})_{\dot{H}^s}|=|(c_\varepsilon \cdot \nabla d_\varepsilon|\delta_{\varepsilon})_{\dot{H}^s}|\leq\||D|^{s}(c_{\varepsilon}\cdot \nabla  d_{\varepsilon})\|_{L^2} \||D|^{s}\delta_{\varepsilon}\|_{L^2} \\
\leq& C_s(\||D|^{s}c_{\varepsilon}\|_{L^{\frac{3}{s}}} \|\nabla  d_{\varepsilon}\|_{L^{\frac{6}{3-2s}}} +\|c_{\varepsilon}\|_{L^{\infty}}\||D|^{s} \nabla  d_{\varepsilon}\|_{L^{2}}+\|M_s(c_{\varepsilon}\,\cdot \nabla d_{\varepsilon})\|_{L^2})\|\delta_{\varepsilon}\|_{\dot{H}^s}\\
\leq&
C_s(\|c_{\varepsilon}\|_{\dot{H}^{\frac{3}{2}}} \| \nabla d_{\varepsilon}\|_{\dot{H}^{s}} +\|  c_{\varepsilon}\|_{\dot{B}^{\frac{3}{2}}_{2,1}}\|  \nabla d_{\varepsilon}\|_{\dot{H}^s}  +\mathbf{C}\|\nabla d_{\varepsilon}\|_{\dot{B}^{s-\beta}_{\frac{2}{1-2\eta},2}}\|c_{\varepsilon}\|_{\dot{B}^{\beta}_{\frac{1}{\eta},2}})\|\delta_{\varepsilon}\|_{\dot{H}^s}\\
\leq&
C_s(\|c_{\varepsilon}\|_{\dot{H}^{\frac{3}{2}}} \| \nabla d_{\varepsilon}\|_{\dot{H}^{s}} +\|  c_{\varepsilon}\|_{\dot{B}^{\frac{3}{2}}_{2,1}}\|  \nabla d_{\varepsilon}\|_{\dot{H}^s}  +\mathbf{C}\|\nabla d_{\varepsilon}\|_{\dot{H}^{s-\beta+3\eta}}\| c_{\varepsilon}\|_{\dot{H}^{\frac{3}{2}+\beta-3\eta}})\|\delta_{\varepsilon}\|_{\dot{H}^s}\\\leq&C_s\|\nabla d_{\varepsilon}\|_{\dot{H}^s}\| c_{\varepsilon}\|_{\dot{B}^{\frac{3}{2}}_{2,1}}\|\delta_{\varepsilon}\|_{\dot{H}^s}\\\leq&
     \frac{1}{2026}\|\nabla d_{\varepsilon}\|_{\dot{H}^s}^2+C_s\| c_{\varepsilon}\|_{\dot{B}^{\frac{3}{2}}_{2,1}}^2\|\delta_{\varepsilon}\|_{\dot{H}^s}^2,
    \end{aligned}
\end{equation}
and
\begin{equation}\label{estG7}
\begin{aligned}
&|(G_7|d_{\varepsilon})_{\dot{H}^s}|=|(c_\varepsilon \otimes\delta_{\varepsilon}|\nabla d_\varepsilon)_{\dot{H}^s}|
\leq\||D|^{s}(c_\varepsilon \otimes\delta_{\varepsilon})\|_{L^2} \||D|^{s}d_{\varepsilon}\|_{L^2} \\
\leq&C_s(\|c_{\varepsilon}\|_{\dot{H}^{\frac{3}{2}}} \| \delta_{\varepsilon}\|_{\dot{H}^{s}} +\|c_{\varepsilon}\|_{{L^{\infty}}}\| \delta_{\varepsilon}\|_{\dot{H}^s}+\mathbf{C}\| \delta_{\varepsilon}\|_{\dot{B}^{s-\beta}_{\frac{2}{1-2\eta},2}}\|c_{\varepsilon}\|_{\dot{B}^{\beta}_{\frac{1}{\eta},2}})\|\nabla d_{\varepsilon}\|_{\dot{H}^s}\\
\leq&C_s(\|c_{\varepsilon}\|_{\dot{H}^{\frac{3}{2}}} \|  \delta_{\varepsilon}\|_{\dot{H}^{s}} +\|  c_{\varepsilon}\|_{\dot{B}^{\frac{3}{2}}_{2,1}}\|  \delta_{\varepsilon}\|_{\dot{H}^s}  +\mathbf{C}\|\delta_{\varepsilon}\|_{\dot{H}^{s-\beta+3\eta}}\| c_{\varepsilon}\|_{\dot{H}^{\frac{3}{2}+\beta-3\eta}})\|\nabla d_{\varepsilon}\|_{\dot{H}^s}\\\leq&
C_s\|\delta_{\varepsilon}\|_{\dot{H}^{s}}\|  c_{\varepsilon}\|_{\dot{B}^{\frac{3}{2}}_{2,1}}\|\nabla d_{\varepsilon}\|_{\dot{H}^s}\\\leq&
    \frac{1}{2026}\|\nabla d_{\varepsilon}\|_{\dot{H}^s}^2+C_s\|\delta_{\varepsilon}\|_{\dot{H}^s}^2\| c_{\varepsilon}\|_{\dot{B}^{\frac{3}{2}}_{2,1}}^2,
    \end{aligned}
\end{equation}
if we choose here also $\beta=3\eta.$

We can treat similarly the last term $G_8$ with this method using non-local derivatives, and the injection $\dot{H}^{s+\frac{1}{2}}\hookrightarrow L^{\frac{6}{3-2(s+\frac12)}}$ (which requires $s+\frac12<\frac32,$ that is $s<1$). 
\begin{equation}\label{estG8}
\begin{aligned}
&|(G_8|d_{\varepsilon})_{\dot{H}^s}|=|(c_\varepsilon \cdot \nabla W_\varepsilon)|d_{\varepsilon})_{\dot{H}^s}|=|(c_\varepsilon \otimes  W_\varepsilon|\nabla d_{\varepsilon})_{\dot{H}^s}|
\leq\||D|^{s}(c_{\varepsilon}\otimes\nabla  W_{\varepsilon})\|_{L^2} \||D|^{s}\nabla d_{\varepsilon}\|_{L^2} \\
\leq&C_s(\||D|^{s}c_{\varepsilon}\|_{L^{3}} \|  W_{\varepsilon}\|_{L^{6}} +\|c_{\varepsilon}\|_{L^{\frac{6}{3-2(s+\frac12)}}}\||D|^{s}W_{\varepsilon}\|_{L^{\frac{3}{s+\frac12}}}+\|M_s(c_{\varepsilon}\,\cdot W_{\varepsilon})\|_{L^2})\|\nabla d_{\varepsilon}\|_{\dot{H}^s}\\
\leq&
C_s(\|c_{\varepsilon}\|_{\dot{H}^{s+\frac{1}{2}}} \| W_{\varepsilon}\|_{L^{6}}+\|  c_{\varepsilon}\|_{\dot{H}^{s+\frac{1}{2}}}\||D|^{s}W_{\varepsilon}\|_{L^{\frac{3}{s+\frac12}}}+\mathbf{C}\|c_{\varepsilon}\|_{\dot{B}^{s-\beta}_{\frac{2}{1-2\eta},2}}\| W_{\varepsilon}\|_{\dot{B}^{\beta}_{\frac1\eta,2}})\|\nabla d_{\varepsilon}\|_{\dot{H}^s}\\
\leq&
C_s(\|c_{\varepsilon}\|_{\dot{H}^{s+\frac{1}{2}}} \| W_{\varepsilon}\|_{L^{6}}+\|  c_{\varepsilon}\|_{\dot{H}^{s+\frac{1}{2}}}\||D|^{s}W_{\varepsilon}\|_{L^{\frac{3}{s+\frac12}}}+\mathbf{C}\|c_{\varepsilon}\|_{\dot{H}^{s+\frac{1}{2}}}\| W_{\varepsilon}\|_{\dot{B}^{\beta}_{\frac{3}{\beta+\frac12},2}})\|\nabla d_{\varepsilon}\|_{\dot{H}^s}\\
\leq&\frac{1}{2026}\|\nabla d_{\varepsilon}\|_{\dot{H}^s}^2+C_s\| c_{\varepsilon}\|_{\dot{H}^{s+\frac{1}{2}}}^2(\| W_{\varepsilon}\|_{L^{6}}^2+\||D|^{s}W_{\varepsilon}\|_{L^{\frac{3}{s+\frac12}}}^2+\| W_{\varepsilon}\|_{\dot{B}^{\beta}_{\frac{3}{\beta+\frac12},2}}^2),
    \end{aligned}
\end{equation}
where this time we choose $\beta\in(0,s)$ and $\eta\in(0,1)$ so that $-\beta+3\eta=\frac12$, for instance
\begin{equation*}
   \begin{cases}
       \beta=\frac14,\\
       \eta=\frac14.
   \end{cases} 
\end{equation*}
\begin{rema}
    With this method we are able to improve the estimate of $G_{14}$ from \cite{FCVSN25}, this is the object of a forthcoming article by the first author.
\end{rema}
Injecting \eqref{est14},\eqref{est48},\eqref{estF6},\eqref{estF7},\eqref{estG7},\eqref{estG8} into \eqref{energy},we obtain
\begin{equation}\label{energy2}
\begin{aligned}
&\frac12 \frac{d}{dt}\|\delta_\varepsilon(t)\|_{\dot{H}^s}^2 + \frac12 \frac{d}{dt}\|d_\varepsilon(t)\|_{\dot{H}^s}^2+\frac34\varepsilon^{\alpha} \|\nabla \delta_\varepsilon(t)\|_{\dot{H}^s}^2  +\frac34\|\nabla d_\varepsilon(t)\|_{\dot{H}^s}^2\\[4pt]\leq& C(\|\delta_\varepsilon(t)\|_{\dot{H}^{\frac12}}^2+\|d_\varepsilon(t)\|_{\dot{H}^{\frac12}}^2)(\|\nabla \delta_\varepsilon(t)\|_{\dot{H}^s}^2 +\|\nabla d_\varepsilon(t)\|_{\dot{H}^s}^2)\\[4pt]&+J_1(t)(\|\delta_\varepsilon(t)\|_{\dot{H}^s}^2+\|d_\varepsilon(t)\|_{\dot{H}^s}^2)+J_2(t),
\end{aligned}
\end{equation}
where 
\begin{equation*}
   J_1(t)=C_s\left( \frac{1}{\varepsilon^{\alpha}} \|\nabla W_\varepsilon\|_{L^3}^2+\frac{1}{\varepsilon^{{3\alpha}}}\|W_\varepsilon\|_{L^6}^4 +\frac{1}{\varepsilon{^{\alpha\frac{s}{1-s}}}}\|W_\varepsilon\|_{L^6}^\frac2{1-s}+\|\nabla c_\varepsilon\|_{\dot{H}^{\frac{1}{2}}}^{2}+\|\nabla  c_{\varepsilon}\|_{\dot{B}^{\frac{3}{2}}_{2,1}}+\|  c_{\varepsilon}\|_{\dot{B}^{\frac{3}{2}}_{2,1}}^2\right),
\end{equation*}
and
\begin{equation*}
\begin{aligned}
  J_2(t)=C_s\left(\|\nabla W_\varepsilon\|_{L^3}^2+\|W_\varepsilon\|^2_{L^6}\|\nabla c_\varepsilon\|_{\dot{H}^{\frac{1}{2}}}^{2s}+\| c_{\varepsilon}\|_{\dot{H}^{s+\frac{1}{2}}}^2\Bigg(\| W_{\varepsilon}\|_{L^{6}}^2+\||D|^{s}W_{\varepsilon}\|_{L^{\frac{3}{s+\frac12}}}^2+\| W_{\varepsilon}\|_{\dot{B}^{\beta}_{\frac{3}{\beta+\frac12},2}}^2\Bigg)\right).   
\end{aligned}
\end{equation*}

Thanks to the fact that $t\leq T_\varepsilon$, we can absorb the first term on the right hand side in \eqref{energy2}, so that using the Gronwall estimates, we obtain that
\begin{equation}\label{enfinal}
\|\delta_\varepsilon(t)\|_{\dot{H}^s}^2+\|d_\varepsilon(t)\|_{\dot{H}^s}^2+\varepsilon^\alpha \int_0^t\|\nabla \delta_\varepsilon(\tau)\|_{\dot{H}^s}^2 d\tau +\int_0^t\|\nabla d_\varepsilon(\tau)\|_{\dot{H}^s}^2 d\tau \leq e^{\int_0^tJ_1(\tau)d\tau} \int _0^tJ_2(\tau)d\tau.
\end{equation}
Next, we estimate the two terms $\int_0^tJ_1(\tau)d\tau$ and $\int _0^tJ_2(\tau)d\tau$, thanks to energy estimates \eqref{enC}:
\begin{equation*}
    \int _0^tJ_1(\tau)d\tau\leq C_s\left(\frac{1}{\varepsilon^{\alpha}}\|\nabla
    W_\varepsilon\|_{L^2L^3}^2+\frac{1}{\varepsilon^{{3\alpha}}}\|W_\varepsilon\|_{L^4L^6}^4 +\frac{1}{\varepsilon{^{\alpha\frac{s}{1-s}}}}\|W_\varepsilon\|_{L^{\frac{2}{1-s}}L^6}^\frac2{1-s}+\| b_{0,\varepsilon}\|_{\dot{H}^{\frac{1}{2}}}^{2}+\|b_{0,\varepsilon}\|_{\dot{B}^{\frac{1}{2}}_{2,1}}\right).
\end{equation*}
\begin{equation*}
\begin{aligned}
    \int _0^tJ_2(\tau)d\tau\leq &C_s\int_0^t \Big[\|\nabla W_\varepsilon\|_{L^3}^2+\|W_\varepsilon\|^2_{L^6}\|\nabla c_\varepsilon\|_{\dot{H}^{\frac{1}{2}}}^{2s}\\&+\| c_{\varepsilon}\|_{\dot{H}^{s+\frac{1}{2}}}^2(\| W_{\varepsilon}\|_{L^{6}}^2+\||D|^{s}W_{\varepsilon}\|_{L^{\frac{3}{s+\frac12}}}^2\| W_{\varepsilon}\|_{\dot{B}^{\beta}_{\frac{3}{\beta+\frac12},2}}^2)\Big]\\\leq& C_s\Big[\|\nabla W_\varepsilon\|_{L^2L^3}^2+\|W_\varepsilon\|^2_{L^{\frac{2}{1-s}}L^6}\|b_{0,\varepsilon}\|_{\dot{H}^{\frac{1}{2}}}^{2s}\\&+\| c_{\varepsilon}\|_{L^4\dot{H}^{s+\frac{1}{2}}}^2(\| W_{\varepsilon}\|_{L^4L^{6}}^2+\||D|^{s}W_{\varepsilon}\|_{L^4L^{\frac{3}{s+\frac12}}}^2+\| W_{\varepsilon}\|_{L^4\dot{B}^{\beta}_{\frac{3}{\beta+\frac12},2}}^2)\Big]\\\leq& C_s\Big[\|\nabla W_\varepsilon\|_{L^2L^3}^2+\|W_\varepsilon\|^2_{L^{\frac{2}{1-s}}L^6}\|b_{0,\varepsilon}\|_{\dot{H}^{\frac{1}{2}}}^{2s}\\&+\|b_{0,\varepsilon}\|_{\dot{H}^{s}}^2(\| W_{\varepsilon}\|_{L^4L^{6}}^2+\||D|^{s}W_{\varepsilon}\|_{L^4L^{\frac{3}{s+\frac12}}}^2+\| W_{\varepsilon}\|_{L^4\dot{B}^{\beta}_{\frac{3}{\beta+\frac12},2}}^2)\Big].
\end{aligned}
\end{equation*}

For the next calculation, we will use the assumption \eqref{stronginitial} several times, and as we assume $m_0\leq1$, we will ignore it except in the exponential. The reason for this approximation will be provided later.
\begin{prop}[\cite{FCVSN25}]\label{semigroup}For any $\sigma \in[0,\frac12+\delta]$, the following estimates hold:
\begin{equation*}
\|c_\varepsilon \cdot \nabla c_\varepsilon\|_{L^1\dot H^{\frac12+\sigma}}
\leq
C\|c_{0,\varepsilon}\|_{H^{\frac12+\sigma}}^2.
\end{equation*}
\end{prop}
 
\begin{prop}(Estimate on $W_{\varepsilon}$)\label{estim:W}
There exist positive constants $\varepsilon_0,\mathrm{C}$ such that if $\delta<\frac13,$ for any $\varepsilon>0$ and any $s\in[\frac12-l\eta_0\delta,\frac12+l\eta_0\delta]$, we have:
\begin{itemize}
    \item $ \|W_\varepsilon\|_{L^4L^6}+  \||D|^{s} W_\varepsilon\|_{L^4L^{\frac{3}{s+\frac12}}}+\||D|^{\beta} W_\varepsilon\|_{L^4L^{\frac{3}{\beta+\frac12}}}\leq \mathrm{C}\varepsilon^{\frac\delta2-\alpha(\frac14-\frac\delta2)}(C_0\varepsilon^{-\gamma} 
+|\ln \varepsilon|)$,
    \item $ \|W_\varepsilon\|_{L^{\frac{2}{1-s}}L^6}\leq\mathrm{C}\varepsilon^{\frac14+\frac \delta2-\frac s2-\alpha(\frac14-\frac\delta2)}(C_0\varepsilon^{-\gamma} 
+|\ln \varepsilon|)$,
    \item $\|\nabla W_\varepsilon\|_{L^2L^3} \leq\mathrm{C}\varepsilon^{\frac\delta2-\alpha(\frac12-\frac\delta2)}(C_0\varepsilon^{-\gamma}+|\ln \varepsilon|)$,
    \item for any small $b>0$ we have
$$\|W_\varepsilon\|_{L^2 L^\infty}\leq C\varepsilon^{\frac{\delta}{2(1+b)}-\alpha(\frac12-\frac{\delta}{2(1+b)})}(C_0\varepsilon^{-\gamma}+|\ln \varepsilon|).$$
\end{itemize}
\end{prop}
\begin{proof}
There are six cases; the estimates in all these cases rely on Lemma \ref{eststri1}. We provide the detailed explanation as follows. Thanks to Proposition $\ref{semigroup}$, we have that

\textbf{Case 1:} For $\|W_\varepsilon\|_{\widetilde{L}^4_t \dot{B}^0_{6,2}}$, we choose $d=0,\,p=4,\,r=6,$ then we compute $\sigma_1=\frac32-\frac12-\frac12+\theta(1-\frac13)=\frac12+\frac32\theta,$ so $\sigma_1=\frac12+\delta \Longleftrightarrow \theta=\frac32\delta$, and from the restrictions on $\theta,p$
 \begin{equation}\label{delta1}
     \begin{cases}
\theta=\frac32\delta\leq1\Longleftrightarrow\delta\leq\frac23 ,\\p\leq\frac{2}{\theta(1-\frac2r)}\Longleftrightarrow\delta\leq\frac12.
     \end{cases}
 \end{equation}
Next thanks to Propositions \ref{defbesov} and \ref{Classicalinjections},
\begin{equation*}
\begin{aligned}
\|W_\varepsilon\|_{L^4L^6}\leq\|W_\varepsilon\|_{L^4\dot{B}^0_{6,2}}\leq\|W_\varepsilon\|_{\widetilde{L}^4_t \dot{B}^0_{6,2}}&\leq C \varepsilon^{\frac\theta3-\alpha(\frac14-\frac\theta3)}(\|u_{0,\varepsilon} \|_{\dot{B}^{\frac12+\delta}_{2,2}} 
+ \|c_\varepsilon\cdot\nabla c_\varepsilon\|_{\widetilde{L}^1_t \dot{B}^{\frac12+\delta}_{2,2}} )\\&\leq C\varepsilon^{\frac\theta3-\alpha(\frac14-\frac\theta3)}(\|u_{0,\varepsilon} \|_{\dot{H}^{\frac12+\delta}} 
+\|b_{0,\varepsilon}\|_{\dot{H}^{\frac12+\delta}}^2 )
\\&\leq C\varepsilon^{\frac\delta2-\alpha(\frac14-\frac\delta2)}(C_0\varepsilon^{-\gamma} 
+|\ln \varepsilon|).
\end{aligned}
\end{equation*}

\textbf{Case 2:} For $\||D|^{s}W_\varepsilon\|_{\widetilde{L}^4_t \dot{B}^0_{\frac{3}{s+\frac12},2}}$, we choose $d=s,\,p=4,\,r=\frac{3}{s+\frac12},\,$ then we compute $\sigma_1=s+\frac32-(s+\frac12)-\frac12+\theta(1-\frac23(s+\frac12))=\frac{1}{2}+\frac{2\theta(1-s)}{3}$, so $\sigma_1=\frac12+\delta \Longleftrightarrow \theta=\frac{3}{2(1-s)}\delta$, and from the restrictions on $\theta,p$
\begin{equation}\label{delta2}
    \begin{cases}
        \theta=\frac{3}{2(1-s)}\delta\leq1\Longleftrightarrow\delta\leq\frac{2(1-s)}{3}, \\p\leq\frac{2}{\theta(1-\frac2r)}\Longleftrightarrow\delta\leq\frac12.
    \end{cases}
\end{equation}
Similarly,
\begin{equation*}
\begin{aligned}
     \||D|^{s} W_\varepsilon\|_{L^4L^{\frac{3}{s+\frac12}}}\leq\||D|^{s} W_\varepsilon\|_{L^4_t \dot{B}^0_{\frac{3}{s+\frac12},2}}&\leq\||D|^{s}W_\varepsilon\|_{\widetilde{L}^4_t \dot{B}^0_{\frac{3}{s+\frac12},2}}\\&\leq C\varepsilon^{\frac{\theta}{3}(1-s)-\alpha(\frac14-\frac{\theta}{3}(1-s))}( \|u_{0,\varepsilon} \|_{\dot{H}^{\frac12+\delta}} 
+\|b_{0,\varepsilon}\|_{H^{\frac12+\delta}}^2 )
\\&\leq C\varepsilon^{\frac{\delta}{2}-\alpha(\frac14-\frac{\delta}{2})}(C_0\varepsilon^{-\gamma} 
+|\ln \varepsilon|).
\end{aligned}
\end{equation*}

\textbf{Case 3:} For $\||D|^{\beta} W_\varepsilon\|_{\widetilde{L}^4_t \dot{B}^\beta_{\frac{3}{\beta+\frac12},2}}$, we choose $d=\beta,\ p=4,\ r=\frac{3}{\beta+\frac12},$ then we compute $\sigma_1=\beta+\frac32-(\beta+\frac12)-\frac12+\theta(1-\frac23(\beta+\frac12))=\frac{1}{2}+\frac{2\theta(1-\beta)}{3}$, so $\sigma_1=\frac12+\delta \Longleftrightarrow \theta=\frac{3}{2(1-\beta)}\delta$, and from the restrictions on $\theta,p$
\begin{equation}\label{delta3}
\begin{cases}
     \theta=\frac{3}{2(1-\beta)}\delta\leq1\Longleftrightarrow\delta\leq\frac{2(1-\beta)}{3},\\p\leq\frac{2}{\theta(1-\frac2r)}\Longleftrightarrow\delta\leq\frac12,
\end{cases}
\end{equation}
and
\begin{equation*}
\begin{aligned}
     \||D|^{\beta} W_\varepsilon\|_{L^4L^{\frac{3}{\beta+\frac12}}}\leq\||D|^{\beta} W_\varepsilon\|_{L^4_t \dot{B}^\beta_{\frac{3}{\beta+\frac12},2}}\leq \||D|^{\beta} W_\varepsilon\|_{\widetilde{L}^4_t \dot{B}^\beta_{\frac{3}{\beta+\frac12},2}}\leq C\varepsilon^{\frac{\delta}{2}-\alpha(\frac14-\frac{\delta}{2})}(C_0\varepsilon^{-\gamma} 
+|\ln \varepsilon|).
\end{aligned}
\end{equation*}

\textbf{Case 4: }For $\|W_\varepsilon\|_{\widetilde{L}^{\frac{2}{1-s}}_t \dot{B}^0_{6,2}}$, we choose $d=0,\ p=\frac{2}{1-s},\ r=6,$ then we compute $\sigma_1=\frac32-\frac12+s-1+\theta(1-\frac13)=\frac{2\theta}{3}+s$, so $\sigma_1=\frac12+\delta \Longleftrightarrow \theta=\frac32(\frac12+\delta-s)$, and from the restrictions on $\theta,p$
\begin{equation}\label{delta4}
\begin{cases}
\theta=\frac32(\frac12+\delta-s)\leq1\Longleftrightarrow\delta\leq \frac16+s,\ \ \ \text{where s is positive,}\\p\leq\frac{2}{\theta(1-\frac2r)}\Longleftrightarrow\delta\leq\frac12.\\
\end{cases}
\end{equation}
We obtain that
\begin{equation*}
\begin{aligned}
\|W_\varepsilon\|_{L^{\frac{2}{1-s}}L^6}\leq\|W_\varepsilon\|_{L^{\frac{2}{1-s}}\dot{B}^0_{6,2}}\leq\|W_\varepsilon\|_{\widetilde{L}^{\frac{2}{1-s}}_t \dot{B}^0_{6,2}}\leq C\varepsilon^{\frac14+\frac \delta2-\frac s2-\alpha(\frac14-\frac\delta2)}(C_0\varepsilon^{-\gamma}+|\ln \varepsilon|).
\end{aligned}
\end{equation*}

\textbf{Case 5:} For $\|\nabla W_\varepsilon\|_{\widetilde{L}^2_t \dot{B}^0_{3,2}},$ we choose $d=1,p=2,r=3,$ then we compute $\sigma_1=1+\frac32-1-1+\theta(1-\frac23)=\frac12+\frac\theta3$, so $\sigma_1=\frac12+\delta \Longleftrightarrow \theta=3\delta$, and from the restriction on $\theta,p$\\
\begin{equation}\label{delta5}
\begin{cases}
\theta=3\delta\leq1\Longleftrightarrow\delta\leq\frac13 \\
 p\leq\frac{2}{\theta(1-\frac2r)}\Longleftrightarrow\delta\leq1.
\end{cases}
\end{equation}
We have that
\begin{equation*}
\begin{aligned}
\|\nabla W_\varepsilon\|_{L^2L^3}\leq\|\nabla W_\varepsilon\|_{L^2_t \dot{B}^0_{3,2}}\leq\|\nabla W_\varepsilon\|_{\widetilde{L}^2_t \dot{B}^0_{3,2}}\leq C\varepsilon^{\frac\delta2-\alpha(\frac12-\frac\delta2)}(C_0\varepsilon^{-\gamma}+|\ln \varepsilon|).
\end{aligned}
\end{equation*}

\textbf{ Case 6 : }first using Besov embedding inequality, and Propositions \ref{defbesov} and \ref{Classicalinjections}
\begin{equation*}
      \|W_\varepsilon\|_{L^2 L^\infty} \leq \|W_\varepsilon\|_{L^2\dot{B}_{\infty, 1}^0} \leq \|W_\varepsilon\|_{\tilde{L}^2 \dot{B}_{\infty, 1}^0},
\end{equation*} we choose $d=0,p=2,r=\infty,p=1$; then we compute $\sigma_1=0+\frac32-1+\theta=\frac12+\theta$, so $\sigma_1=\frac12+\delta \Longleftrightarrow \theta=3\delta$, and we obtain
\begin{equation*}
 \|W_\varepsilon\|_{\tilde{L}^2 \dot{B}_{\infty, 1}^0}\leq C\varepsilon^{\frac\theta2-\alpha(\frac12-\frac\theta2)}(\|u_{0,\varepsilon} \|_{\dot{B}^{\frac12+\theta}_{2,1}} 
+ \|c_\varepsilon\cdot\nabla c_\varepsilon\|_{\widetilde{L}^1_t \dot{B}^{\frac12+\theta}_{2,1}}).  
\end{equation*}
Thanks to Proposition \ref{interpolation}, with the same notations and arguments as in the proof of Proposition 5.2 from \cite{FCVSN25} [($(\alpha,\beta)=(a\theta,b\theta)$)]. We have that
\begin{equation}\label{u0ab}
\|u_{0,\varepsilon}\|_{\dot{B}^{\frac12+\theta}_{2,1}} \leq C_{a,b} \|u_{0,\varepsilon}\|_{\dot{H}^{\frac12+(1-a)\theta}}^{\frac{b}{a+b}}\|u_{0,\varepsilon}\|_{\dot{H}^{\frac12+(1+b)\theta}}^{\frac{a}{a+b}},
\end{equation}
here $a,b>0,\theta\in(0,1)$ such that
\begin{equation*}
    \begin{cases}
    (1-a)\theta=-\delta,\\       (1+b)\theta=\delta,
    \end{cases}
\end{equation*}
then for any $b>0$ and small enough(to be precised later), we choose $\theta=\frac{\delta}{1+b},$ and $a=b+2,$ \eqref{u0ab} implies
\begin{equation}\label{u0bmu}
\|u_{0,\varepsilon}\|_{\dot{B}_{2,1}^{\frac12+\theta}} \leq C_{b} \|u_{0,\varepsilon}\|_{\dot{H}^{\frac12-\delta}}^\frac{b}{a+b} \|u_{0,\varepsilon}\|_{\dot{H}^{\frac12+\delta}}^\frac{a}{a+b} \leq C_{b}(C_0 \varepsilon^{-\gamma})^\frac{b}{a+b}(C_0 \varepsilon^{-\gamma})^\frac{a}{a+b}\leq C_{b}C_0 \varepsilon^{-\gamma}.
\end{equation}
Similarly, thanks to Propositions \ref{interpolation} and \ref{semigroup}, we obtain
\begin{equation}\label{cbmu}
\begin{aligned}
\|c_\varepsilon\cdot\nabla c_\varepsilon\|_{\widetilde{L}^1_t \dot{B}^{\frac12+\theta}_{2,1}}\leq \|c_\varepsilon\cdot\nabla c_\varepsilon\|_{L^1_t \dot{B}^{\frac12+\theta}_{2,1}}&\leq C_{b} \|c_\varepsilon\cdot\nabla c_\varepsilon\|_{L^1_t\dot{H}^{\frac12-\delta}}^\frac{b}{a+b}\|c_\varepsilon\cdot\nabla c_\varepsilon\|_{L^1_t\dot{H}^{\frac12-\delta}}^\frac{a}{a+b}\\&\leq C_b \|b_{0,\varepsilon}\|_{\dot{H}^{\frac12+\delta}}^2\leq C_b|\ln \varepsilon|.
\end{aligned}
\end{equation}
Then combining \eqref{u0bmu} and \eqref{cbmu}, $\theta$ is $\frac{\delta}{1+b}$, so we have 
\begin{equation*}
\begin{aligned}
    \|W_\varepsilon\|_{\tilde{L}^2 \dot{B}_{\infty, 1}^0}&\leq C\varepsilon^{\frac{\delta}{2(1+b)}-\alpha(\frac12-\frac{\delta}{2(1+b)})}(C_0\varepsilon^{-\gamma}+|\ln \varepsilon|).
\end{aligned} 
\end{equation*}
\end{proof}
\begin{rema}\label{rangedelta}
By examining all the constraints on $\delta$ [\eqref{delta1},\eqref{delta2},\eqref{delta3},\eqref{delta4},\eqref{delta5}], we obtain that 
\begin{equation}
\delta\leq\frac{2(1-s)}{3} , 
\end{equation}
then the smallest $\delta$ is attained when $s=\frac12+l\eta_0\delta$, and we have 
\begin{equation}\label{delta}
\delta\leq\frac{2(1-s)}{3}\leq\frac{1-2l\eta_0\delta}{3}\Longleftrightarrow\delta\leq\frac{1}{3+2l\eta_0}<\frac13.   
\end{equation}
It says that we can choose $\delta<\frac13,$ but we have to choose $l$ small enough:
\begin{equation}\label{l1}
  0<l\leq \frac{1-3\delta}{\delta-2\gamma},  
\end{equation}
so that \eqref{delta} holds true. 
Because $l\leq 1,\,\delta<\frac13,\,\eta_0\leq\frac12$, we also obtain $s\in[\frac13,\frac23].$
\end{rema}

Thanks to Proposition $\ref{estim:W}$, we obtain for any $t\leq T_{\varepsilon},$ for any $s\in[\frac12-l\eta_0\delta,\frac12+l\eta_0\delta]$
\begin{equation*}
\begin{aligned}
&\|\delta_\varepsilon(t)\|_{\dot{H}^s}^2 + \|d_\varepsilon(t)\|_{\dot{H}^s}^2+\int_0^t(\varepsilon^{\alpha} \|\nabla \delta_\varepsilon(\tau)\|_{\dot{H}^s}^2  +\|\nabla d_\varepsilon(\tau)\|_{\dot{H}^s}^2)\,d \tau\\
\leq& \exp\Bigg[
C_s\Big(\frac{1}{\varepsilon^{\alpha}}\|\nabla W_\varepsilon\|_{L^2L^3}^2+\frac{1}{\varepsilon^{3\alpha}}\|W_\varepsilon\|_{L^4L^6}^4 +\frac{1}{\varepsilon^{\alpha(\frac{s}{1-s})}}
\|W_\varepsilon\|_{L^{\frac{2}{1-s}}L^6}^{\frac{2}{1-s}}
+\|b_{0,\varepsilon}\|_{\dot H^{\frac12}}^{2}
+\|b_{0,\varepsilon}\|_{\dot B^{\frac12}_{2,1}}
\Big)
\Bigg]\\
&\times C_s
\Bigg[
\|\nabla W_\varepsilon\|_{L^2L^3}^2
+\|W_\varepsilon\|_{L^{\frac{2}{1-s}}L^6}^{2}
  \|b_{0,\varepsilon}\|_{\dot H^{\frac12}}^{2s}\\
&\,\,\,\,\,\,\,\,\,+\|b_{0,\varepsilon}\|_{\dot H^{s}}^2
\big(
\|W_\varepsilon\|_{L^4L^6}^2
+\||D|^s W_\varepsilon\|_{L^4L^{\frac{3}{s+\frac12}}}^2
+\|W_\varepsilon\|_{L^4\dot B^\beta_{\frac{3}{\beta+\frac12},2}}^2
\big)
\Bigg]\\
\le&
\Bigg[\exp (2C_sm_0^2|\ln \varepsilon|)\Bigg]\Bigg[\exp C_s\Big(
\varepsilon^{2(\frac\delta2-\alpha(\frac12-\frac\delta2))-\alpha}
(C_0\varepsilon^{-\gamma}+|\ln\varepsilon|)^2+\varepsilon^{4(\frac\delta2-\alpha(\frac14-\frac\delta2))-3\alpha}
(C_0\varepsilon^{-\gamma}+|\ln\varepsilon|)^4\\
&+\varepsilon^{\frac{2}{1-s}(\frac14+\frac\delta2-\frac s2-\alpha(\frac14-\frac\delta2))
-\alpha(\frac{s}{1-s})}
(C_0\varepsilon^{-\gamma}+|\ln\varepsilon|)^{\frac{2}{1-s}}\Big)
\Bigg]  \\
&\times
\Bigg[C_s \Big(
\varepsilon^{2(\frac\delta2-\alpha(\frac12-\frac\delta2))}
\big(C_0\varepsilon^{-\gamma}+|\ln\varepsilon|\big)^2
+\varepsilon^{2(\frac14+\frac\delta2-\frac s2-\alpha(\frac14-\frac\delta2))}
\big(C_0\varepsilon^{-\gamma}+|\ln\varepsilon|\big)^{2}
(|\ln\varepsilon|^{s})  \\
&\qquad
+|\ln\varepsilon|
\varepsilon^{2(\frac{\delta}{2}-\alpha(\frac14-\frac{\delta}{2}))}
\big(C_0\varepsilon^{-\gamma}+|\ln\varepsilon|\big)^2 \Big)
\Bigg]\overset{def}{=}\mathcal{ABC}.
\end{aligned}
\end{equation*}
where $\|b_{0,\varepsilon}\|_{\dot{B}_{2,1}^{\frac12}}\leq \|b_{0,\varepsilon}\|_{{H}^{\frac12+\delta}}\leq |\ln\varepsilon|^{\frac12},$ and obviously $|\ln\varepsilon|^{\frac12}+|\ln\varepsilon|\leq 2|\ln\varepsilon|$ is true for $\varepsilon\leq\frac1e$.

First, we have
\begin{equation*}
\mathcal{A}=exp(2C_sm_0^2|\ln\varepsilon|)=\varepsilon^{-2C_sm_0^2}. 
\end{equation*}

\begin{rema}
Actually $m_0\leq \sqrt{\frac{l\eta_0\delta}{2C_s}} \Longleftrightarrow \varepsilon^{-2C_sm_0^2}\leq \varepsilon^{-l\eta_0\delta}$. Since $l,\ \delta$ and $\eta_0$ are small, we are sure that $m_0\leq 1.$
\end{rema}
Now, as $\mathcal{A}\leq \varepsilon^{-l\eta_0\delta}$, and if we assume $\varepsilon>0$ is so small that $|\ln\varepsilon|\leq\varepsilon^{-l\eta_0\delta}$, which determines $\varepsilon_0$, we can write
\begin{equation*}
C_0\varepsilon^{-\gamma}+|\ln\varepsilon|\leq C_0\varepsilon^{-\gamma}+\varepsilon^{-l\eta_0\delta}\leq C_0\varepsilon^{-\gamma-l\eta_0\delta}\leq C_0\varepsilon^{-\frac\delta2+(1-l)\eta_0\delta}, 
\end{equation*} 
by returning to the notations $\gamma=\frac\delta2-\eta_0\delta.$
And we can bound $\mathcal{B}$ and $\mathcal{C}$ as follows:
\begin{equation*}
\begin{aligned}
\mathcal{B}
=&\exp C_s\Big(
\varepsilon^{2(\frac\delta2-\alpha(\frac12-\frac\delta2))-\alpha}
(C_0\varepsilon^{-\gamma}+|\ln\varepsilon|)^2+ \varepsilon^{4(\frac\delta2-\alpha(\frac14-\frac\delta2))-3\alpha}
(C_0\varepsilon^{-\gamma}+|\ln\varepsilon|)^4 \\
&+ \varepsilon^{\frac{2}{1-s}\left(\frac14+\frac \delta2-\frac s2-\alpha\left(\frac14-\frac\delta2\right)\right)-\alpha\frac{s}{1-s}}
(C_0\varepsilon^{-\gamma}+|\ln\varepsilon|)^{\frac{2}{1-s}}
\Big)\\\leq& \exp \Big(C_s\big(\varepsilon^{2\left(\frac\delta2-\alpha\left(1-\frac\delta2\right)-\gamma-l\eta_0\delta\right)}+\varepsilon^{4\left(\frac\delta2-\alpha\left(1-\frac\delta2\right)-\gamma-l\eta_0\delta\right)}+\varepsilon^{\frac{2}{1-s}\left(\frac14+\frac \delta2-\frac s2-\alpha\left(\frac14+\frac s2-\frac\delta2\right)-\gamma-l\eta_0\delta\right)}\big)\Big),
\end{aligned}
\end{equation*}
because $s\in[\frac12-l\eta_0\delta,\frac12+l\eta_0\delta]$, we are forced to keep the minimal power of $\varepsilon$, which corresponds to $s=\frac12+l\eta_0\delta$. From Remark \ref{rangedelta}, we know that $\frac2{1-s}\geq3$, so we can write
\begin{equation}\label{B1}
    \mathcal{B}\leq \exp \left(C\left(\varepsilon^{2\left((1-l)\eta_0\delta-\alpha(1-\frac\delta2)\right)}+\varepsilon^{4\left((1-l)\eta_0\delta-\alpha(1-\frac\delta2)\right)}+\varepsilon^{3\left((1-\frac32l)\eta_0\delta-\alpha(\frac12+\frac12l\eta_0\delta-\frac\delta2)\right)}\right)\right).
\end{equation}
Similarly, we obtain
\begin{equation}\label{C1}
\begin{aligned}
\mathcal{C}=&C_s \Big(\varepsilon^{2(\frac\delta2-\alpha(\frac12-\frac\delta2))}
\left(C_0\varepsilon^{-\gamma}+|\ln\varepsilon|\right)^2 + \varepsilon^{2
(\frac14+\frac \delta2-\frac s2-\alpha(\frac14-\frac\delta2))}
\left(C_0\varepsilon^{-\gamma}
+|\ln\varepsilon|\right)^{2}
|\ln \varepsilon|^s\\&+|\ln\varepsilon|\,
\varepsilon^{2(\frac{\delta}{2}-\alpha(\frac14-\frac{\delta}{2}))}
\left(C_0\varepsilon^{-\gamma}+|\ln\varepsilon|\right)^2\Big)\\\leq&C_s \left(\varepsilon^{2\left(\frac\delta2-\alpha(\frac12-\frac\delta2)-\gamma-l\eta_0\delta)\right)}+\varepsilon^{2\left(\frac14+\frac \delta2-\frac s2-\alpha(\frac14-\frac\delta2)-\gamma-\frac32l\eta_0\delta\right)}+\varepsilon^{2\left(\frac{\delta}{2}-\alpha(\frac14-\frac{\delta}{2})-\gamma-\frac32l\eta_0\delta\right)}\right),
\\\leq&C\left(\varepsilon^{2\left((1-l)\eta_0\delta-\alpha(\frac12-\frac\delta2)\right)}+\varepsilon^{2\left((1-2l)\eta_0\delta-\alpha(\frac14-\frac\delta2)\right)}+\varepsilon^{2\left((1-\frac32l)\eta_0\delta-\alpha(\frac14-\frac\delta2)\right)}\right).
\end{aligned}
\end{equation}
Combining the expressions of $\mathcal{A},\mathcal{B},\mathcal{C},$ we obtain that \eqref{enfinal} turns into: 
\begin{equation}\label{ABC}
\begin{aligned}
&\|\delta_\varepsilon(t)\|_{\dot{H}^s}^2+\|d_\varepsilon(t)\|_{\dot{H}^s}^2+\varepsilon^\alpha \int_0^t\|\nabla \delta_\varepsilon(\tau)\|_{\dot{H}^s}^2 d\tau +\int_0^t\|\nabla d_\varepsilon(\tau)\|_{\dot{H}^s}^2 d\tau\\\leq& C\exp \Big(C^{'}\big(\varepsilon^{2((1-l)\eta_0\delta-\alpha(1-\frac\delta2))}+\varepsilon^{4((1-l)\eta_0\delta-\alpha(1-\frac\delta2))}+\varepsilon^{3((1-\frac32l)\eta_0\delta-\alpha(\frac12+\frac12l\eta_0\delta-\frac\delta2))}\big)\Big)\\&\times\left(\varepsilon^{2((1-\frac32l)\eta_0\delta-\alpha(\frac12-\frac\delta2))}+\varepsilon^{2((1-\frac52l)\eta_0\delta-\alpha(\frac14-\frac\delta2))}+\varepsilon^{2((1-2l)\eta_0\delta-\alpha(\frac14-\frac\delta2))}\right)
\end{aligned}
\end{equation} where $C^{'}$ depend on $\delta,C_0.$
 Then, considering the powers of $\varepsilon$ in \eqref{B1}, if they satisfy
\begin{equation}\label{index1}
\begin{cases}
(1-l)\eta_0\delta-\alpha(1-\frac\delta2)\geq0,\\
(1-\frac32l)\eta_0\delta-\alpha(\frac12+\frac12l\eta_0\delta-\frac\delta2)\geq0,\\
\end{cases} 
\end{equation}
then we obtain $\mathcal{B}$ is bounded by $e^{3C}$; and if we denote $j$ as the minimum of the three powers of $\varepsilon$ appearing in the final line in \eqref{ABC}, then we obtain for $s=\frac12$
\begin{equation*}
\|\delta_\varepsilon(t)\|_{\dot{H}^{\frac12}}^2+\|d_\varepsilon(t)\|_{\dot{H}^{\frac12}}^2+\varepsilon^\alpha \int_0^t\|\nabla \delta_\varepsilon(\tau)\|_{\dot{H}^{\frac12}}^2 d\tau +\int_0^t\|\nabla d_\varepsilon(\tau)\|_{\dot{H}^{\frac12}}^2 d\tau \leq 3Ce^{3C}\varepsilon^{j},
\end{equation*} 
in order to close the bootstrap argument, we need that (where $\varepsilon$ is small enough)
\begin{equation}\label{bootj}
    3\varepsilon^j\leq (\frac{1}{8C}\varepsilon^\alpha)^2,
\end{equation}
so we need to ask $j>2\alpha$, that is
\begin{equation}\label{index2}
\begin{cases}
(1-\frac32l)\eta_0\delta-\alpha(\frac12-\frac\delta2)>\alpha,\\
((1-\frac52l)\eta_0\delta-\alpha(\frac14-\frac\delta2)>\alpha,\\
(1-2l)\eta_0\delta-\alpha(\frac14-\frac\delta2)>\alpha.
\end{cases}\Longleftarrow\,\,\,\, \alpha< \frac{(1-\frac52l)\eta_0\delta}{\frac32-\frac\delta2}
\end{equation}

Examining these formulas in \eqref{index1} and \eqref{index2}, since $l<1,\delta<1,\eta_0<1$, so $\frac12+\frac12l\eta_0\delta <\frac32$, all the previous conditions are fulfilled when
\begin{equation}\label{l2}
(1-\frac52l)(\frac\delta2-\gamma)-\alpha(\frac32-\frac{\delta}{2})>0.    
\end{equation}
In fact, we can ask
\begin{equation*}
(\frac\delta2-\gamma)-\alpha(\frac32-\frac{\delta}{2})>0,    
\end{equation*} and we choose $l\in(0,1)$ small enough so that \eqref{l2} holds. Here, \eqref{l2} was one of the smallest conditions on $l>0,$ and \eqref{l1} is another one.

In \eqref{delta}, there exist $l,\eta_0>0$ that are small enough; we also ask
\begin{equation}\label{kl}
1-\frac52l\geq k\longrightarrow1, \end{equation} such that for any $\delta<\frac13,$ we directly have $$\alpha<\frac{\frac\delta2-\gamma}{\frac32-\frac{\delta}{2}}\leq \frac{\frac\delta2}{\frac32-\frac{\delta}{2}}\leq\frac18.$$

Then from \eqref{bootj}, we obtain that for any $\varepsilon \in (0,\varepsilon_0]$ 
\begin{equation}
\begin{aligned}
&\|\delta_\varepsilon(t)\|_{\dot{H}^{\frac12}}^2+\|d_\varepsilon(t)\|_{\dot{H}^{\frac12}}^2+\varepsilon^\alpha \int_0^t\|\nabla \delta_\varepsilon(\tau)\|_{\dot{H}^{\frac12}}^2 d\tau +\int_0^t\|\nabla d_\varepsilon(\tau)\|_{\dot{H}^{\frac12}}^2 d\tau \\\leq& \varepsilon ^{2[k(\frac\delta2-\gamma)-\alpha(\frac12-\frac\delta2)]} \leq(\frac{1}{8C}\varepsilon^{\alpha})^2, 
\end{aligned}
\end{equation}
this contradicts the definition of $T_\varepsilon$, and we have $T_\varepsilon=T_\varepsilon^*=\infty.$

Thanks to the propagation of regularity, finally, we obtain that 
for any $s\in[\frac12-l\eta_0\delta,\frac12+l\eta_0\delta],\delta\in(0,\frac13),\alpha\in(0,\frac18)$, \eqref{enfinal} is
\begin{equation}
\|\delta_\varepsilon(t)\|_{\dot{H}^s}^2+\|d_\varepsilon(t)\|_{\dot{H}^s}^2+\varepsilon^\alpha \int_0^t\|\nabla \delta_\varepsilon(\tau)\|_{\dot{H}^s}^2 d\tau +\int_0^t\|\nabla d_\varepsilon(\tau)\|_{\dot{H}^s}^2 d\tau \leq \varepsilon^{2[k(\frac\delta2-\gamma)-\alpha(\frac12-\frac\delta2)]},
\end{equation}
then we obtain
\begin{equation}
\|D_\varepsilon\|_{L^2L^\infty}\leq \|D_\varepsilon\|_{L^2\dot{B}^{\frac32}_{2,1}}\leq\|D_\varepsilon\|^{\frac12}_{L^2\dot{H}^{\frac32-l\eta_0\delta}}\|D_\varepsilon\|^{\frac12}_{L^2\dot{H}^{\frac32+l\eta_0\delta}}\leq\varepsilon^{k(\frac\delta2-\gamma)-\alpha(1-\frac\delta2)}.
\end{equation}
And applying Proposition \ref{Wvar} again, we obtain
\begin{equation}
\begin{aligned}
\|\varepsilon^{\frac\alpha2}u_\varepsilon\|_{L^2L^\infty}&=\|\varepsilon^{\frac\alpha2}\delta_\varepsilon+\varepsilon^{\frac\alpha2}W_\varepsilon\|_{L^2L^\infty}\\&\leq\varepsilon^{k(\frac\delta2-\gamma)-\alpha(\frac12-\frac\delta2)}+\varepsilon^{\frac\alpha2}\cdot\varepsilon^{\frac{\delta}{2(1+b)}-\alpha(\frac12-\frac{\delta}{2(1+b)})}(C_0\varepsilon^{-\gamma}+|\ln \varepsilon|)\\&\leq
\varepsilon^{k(\frac\delta2-\gamma)-\alpha(\frac12-\frac\delta2)}+
\varepsilon^{\frac{\delta}{2(1+b)}+\alpha\frac{\delta}{2(1+b)}}(\varepsilon^{-\gamma-l\eta_0\delta})
\\&\leq
\varepsilon^{k(\frac\delta2-\gamma)-\alpha(\frac12-\frac\delta2)}+
\varepsilon^{\frac{\delta}{2(1+b)}+\alpha\frac{\delta}{2(1+b)}}(\varepsilon^{-\gamma-l(\frac\delta2-\gamma)})\\&\leq2
\varepsilon^{k(\frac\delta2-\gamma)-\alpha(\frac12-\frac\delta2)}
\end{aligned}
\end{equation}
 the second term on the penultimate line needs to satisfy 
 \begin{equation}\label{kdela}
(k+l)(\frac\delta2-\gamma)\leq \frac{\delta(1+\alpha)}{2(1+b)}+\alpha(\frac12-\frac\delta2)-\gamma,     
 \end{equation}
because of \eqref{kl}, we know that $k+l\leq 1-\frac32l\longrightarrow 1,$ so \eqref{kdela} becomes
\begin{equation}\label{b}
 \left(\frac\delta2-\frac{\delta}{2(1+b)}\right)(1+\alpha)\leq \frac\alpha2,   
\end{equation}
and \eqref{b} is implied by 
\begin{equation}\label{b0}
  0<b\leq\frac{\alpha}{\delta-\alpha(1-\delta)}. 
\end{equation} 
\begin{rema}
Indeed, since
\begin{equation*}
1-\delta \leq \frac{3}{2}-\frac{\delta}{2},    
\end{equation*} is always true,
we obtain
\begin{equation*}
\alpha(1-\delta)
\leq
\alpha\left(\frac{3}{2}-\frac{\delta}{2}\right)
< \frac{\delta}{2}
< \delta.   
\end{equation*}
Therefore,
\begin{equation*}
\delta-\alpha(1-\delta)>0.    
\end{equation*}
\end{rema}
Since \eqref{b0} always holds, then we obtain
\begin{equation*}
    \varepsilon^{\frac{\delta}{2(1+b)}+\alpha\frac{\delta}{2(1+b)}}(\varepsilon^{-\gamma-l(\frac\delta2-\gamma)})\leq \varepsilon^{k(\frac\delta2-\gamma)-\alpha(\frac12-\frac\delta2)},
\end{equation*}
which yields Theorem \ref{mainthe22}.
\clearpage
\onecolumn

\end{document}